\documentclass[11pt]{amsart}
\usepackage{texmac}
\operators{div}
\long\def\comment#1\endcomment{}

\def\shazero{\sha\rlap{${}^{\circ}$}}
\def\smallshazero{\smallsha^{\circ}}

\operators{Stab,Sel,Gal,Om,Irr,Disc}
\calsymbols{c}{C,K,O,M}
\bbsymbols{}{G}
\calsymbols{}{F,K,L}

\def\omegabar{\overline{\omega}}

\let\ge\geqslant
\let\le\leqslant

\DeclareSymbolFont{bbold}{U}{bbold}{m}{n}
\DeclareMathSymbol{\bbmu}{\mathord}{bbold}{"16}

\begin{document}

\title[Growth of $\smallsha$ in towers for isogenous curves]{Growth of $\sha$ in towers for isogenous curves}

\author{Tim and Vladimir Dokchitser}
\date{Christmas Day 2012}
\address{Dept of Mathematics, University Walk, Bristol BS8 1TW, United Kingdom}
\email{tim.dokchitser@bristol.ac.uk}
\address{Emmanuel College, Cambridge CB2 3AP, United Kingdom}
\email{v.dokchitser@dpmms.cam.ac.uk}

\begin{abstract}
We study the growth of $\smallsha$ and $p^\infty$-Selmer groups 
for isogenous abelian varieties in towers 
of number fields, with an emphasis on elliptic curves. 
The growth types are usually exponential, 
as in the setting of `positive $\mu$-invariant' in Iwasawa theory
of elliptic curves. The towers we consider are
$p$-adic and $l$-adic Lie extensions for $l\ne p$,
in particular cyclotomic and other $\Z_l$-extensions.
\end{abstract}

\llap{.\hskip 10cm} \vskip -14mm
\maketitle


\section{Introduction}

The algebraic side of the Iwasawa theory of elliptic curves is 
concerned with the study of
the structure of Selmer groups in cyclotomic $\Zp$-extensions 
of $\Q$, and other towers of number fields. 
%
The aim of the present paper is to systematically study 
the behaviour of Selmer groups for 
isogenous elliptic curves $E, E'$ or abelian varieties.
The isogeny makes it possible to bypass Iwasawa theory
and, in particular, avoid any assumptions on the reduction types.
Moreover, it allows us to work with $p^\infty$-Selmer groups in general 
\hbox{$l$-adic} towers, both for $l\!=\!p$ and $l\!\ne\! p$.
For instance, we construct elliptic cur\-ves whose $p$-primary part of 
the Tate-Shafarevich group goes off to infinity in all 
$l$-cyc\-lotomic extensions of~$\Q$,
in contrast to Washington's theorem that the $p$-part of the ideal
class group is bounded in these extensions for $l\ne p$. 

We show that in the $n$th layer of the $p$-cyclotomic tower of $\Q$,
the quotient
$|\sha_E[p^\infty]|/|\sha_{E'}[p^\infty]|$ (if finite)
is $p^{\,\mu p^n+O(1)}$, as though it came from 
an Iwasawa module with $\lambda$-invariant 0 and $\mu$-invariant $\mu$,
except that $\mu$ may be fractional when $E$ has potentially supersingular 
reduction at $p$. 
Our formula for $\mu$ is explicit and surprisingly simple, and there 
is similar behaviour in other $p$-adic and $l$-adic towers.
It would be interesting to understand~the structure theory of the associated 
Selmer groups that gives rise to such growth.

Our main results for $\Z_l$-extensions and general Lie groups are as follows:

\def\sfrac#1#2{\kern.05em\raise.5ex\hbox{\rm\smaller[2]#1}%
  \kern-.2em/\kern-.1em\lower.25ex\hbox{\rm\smaller[2]#2}}
    
\begin{theorem}
\label{imain}
Let $\cup_n \Q(l^n)$ be the $l$-cyclotomic tower,
$p$ a prime and $E, E'$ two isogenous elliptic curves over $\Q$. 
Then for all large enough $n$,
$$
  \frac{|\shazero_{E/\Q(l^n)}[p^\infty]|}{|\shazero_{E'\!/\Q(l^n)}[p^\infty]|}
     = p^{\,\mu l^n + \nu + \epsilon(n)},
$$
for some $\nu\in \Z$, some $|\epsilon(n)|\le 8\sfrac12$ and $\mu\in\frac1{12}\Z$ given by
$$
  \mu = \ord_p\frac{\Omega_{E'}}{\Omega_E} + 
  \left\{\begin{array}{llllll}
     0&{\text{if $\,l\ne p\>$ or $\>\ord_p(j_E)<0$,}}\\[3pt]
     {\frac1{12}{\ord_p(\frac{\Delta_{E'}}{\Delta_E})}}&{\text{if $\,l=p\>$ and $\>\ord_p(j_E)\ge 0$.}}\cr
  \end{array}
  \right.
$$
If $l\ne p$ or $l$ does not divide the degree of the isogeny $E\to E'$, then
$\epsilon(n)=0$.
\end{theorem}

\noindent
Here and throughout the paper 
$\Q(l^n)$ denotes the degree $l^n$-extension of $\Q$ in the
cyclotomic $\Z_l$-extension. 
We write $\Delta_E, \Delta_{E'}$ for the minimal discriminants%
\footnote{If the base field is not $\Q$, there may be no global minimal model.
We then regard $\Delta_E, \Delta_{E'}$ as ideals that have minimal valuation 
at every prime.} 
and $j_E,j_{E'}$ for the $j$-invariants of the two curves,
$\Omega_E, \Omega_{E'}$ for the Birch--Swinnerton-Dyer periods (see \S\ref{sper}),
and $\shazero\ $ for the Tate-Shafarevich group
$\sha$ modulo its divisible part $\sha^{\div}$. 
Write also $\Sel_{p^\infty}$ for the $p^\infty$-Selmer group, 
$\Sel_{p^\infty}^{\div}$ for its divisible part,
$\rk_p$ for its $\Zp$-corank\footnote{Thus, $\rk_p A/K=\rk A/K+t$ 
if $\smallsha^{\div}_{A/K}\iso(\Q_p/\Z_p)^t$. Of course, conjecturally, $t=0$,
$\smallsha=\smallshazero$ and $\smash{\Sel_{p^\infty}^{\div}(A/K)}\iso A(K)\tensor\Q_p/\Z_p$. 
In any case, $\smallshazero$ is a torsion abelian group
all of whose $p$-primary parts are finite.}, 
$c_v$ for the Tamagawa number at $v$, and $f_{K_v/\Q_p}$ for the residue
degree of $K_v/\Q_p$. The big $O$ notation refers to the parameter $n$.

\par\noindent\vskip-1pt\par\noindent

\begin{theorem}
\label{iZlmain}
Let $K$ be a number field, $p$ a prime number, 
and $K_\infty\!=\!\cup_n K_n$ a $\Zl$-extension of $K$,
with $[K_n:K]=l^n$. Let $\phi: E\to E'$ be an isogeny of 
elliptic curves over~$K$, with dual isogeny $\phi^t$. Then
$$
\frac{|\Sel^{\div}_{p^\infty}(E/K_n)[\phi]|}
     {|\Sel^{\div}_{p^\infty}(E'\!/K_n)[\phi^t]|} 
\frac{|\shazero_{E/K_n}[p^\infty]|}{|\shazero_{E'\!/K_n}[p^\infty]|} 
  = 
  p^{\,\mu l^n+O(1)},
  \quad 
  \mu=\ord_p \bigl(\frac{\Omega_{E'\!/K}}{\Omega_{E/K}}\bigr)+\sum_{v} \mu_v,
$$
where the sum is taken over primes $v$ of bad reduction for $E/K$, and
\begingroup\smaller[1]
$$
  \quad\mu_v = \left\{
    \begin{array}{llll}
    \ord_p\frac{c_v(E'\!/K)}{c_v(E/K)} & \text{if } v\text{ is totally split in $K_\infty/K$},& \cr
    \frac{f_{K_v/\Q_p}}{12}{\,\ord_v(\frac{\Delta_{E'}}{\Delta_E})}& \text{if $l\!=\!p,$\ $v|p$ is ramified in $K_\infty/K$
      and $\ord_v j_E\!\ge\!0$},\cr
    0 & \text{otherwise.} \\[3pt]
    \end{array}
  \right.
$$
\endgroup
If $\rk_p E/K_n$ is bounded, then 
\smash{$
  \frac{|\smallshazero_{E/K_n}[p^\infty]|}{|\smallshazero_{E'\!/K_n}[p^\infty]|}
     = p^{\,\mu l^n + O(1)} 
$}
as well.
\end{theorem}

\vskip 1mm

\begin{theorem}
\label{iLie}
Suppose $K$ is a number field and $K_\infty/K$ a Galois extension 
whose Galois group is a $d$-dimensional $l$-adic Lie group; 
write $K_n/K$ for its $n$th layer in the natural $l$-adic Lie filtration. 
Let $p$ be a prime number, and $\phi: A\to A'$ an isogeny 
of abelian varieties over $K$, with dual $\phi^t: A'^t\to A^t$. 

(1) If $A$ is an elliptic curve, then there is $\mu\in\Q$ such that
$$
\frac{|\Sel^{\div}_{p^\infty}(A/K_n)[\phi]|}
     {|\Sel^{\div}_{p^\infty}(A'\!/K_n)[\phi^t]|} 
\frac{|\shazero_{A/K_n}[p^\infty]|}{|\shazero_{A'\!/K_n}[p^\infty]|} 
\!=\! 
p^{\,\mu l^{dn} \!+\! O(l^{(d-1)n})}.
$$

(2) If either $A, A'$ are semistable abelian varieties or they are
elliptic cur\-ves that do not have additive potentially supersingular 
reduction at primes~$v|p$ that are infinitely ramified in $K_\infty/K$,
then there are constants \hbox{$\mu_1,...,\mu_{d-1}\!\in\!\Q$} such that for all 
sufficiently large $n$,
$$
\frac{|\Sel^{\div}_{p^\infty}(A/K_n)[\phi]|}
     {|\Sel^{\div}_{p^\infty}(A'^t\!/K_n)[\phi^t]|} 
\frac{|\shazero_{A/K_n}[p^\infty]|}{|\shazero_{A'\!/K_n}[p^\infty]|} 
= p^{\,\mu l^{dn} + \mu_1 l^{(d-1)n} + \ldots + \mu_{d-1}l^n + O(1)}.
$$
If $A(K_n)[p^\infty]$ is bounded, $O(1)$ may be replaced by a constant $\mu_d\in\Q$.

(3) $\vphantom{\int^{X^A}}$
If $\rk_p A/K_n=O(l^{(d-1)n})$ in (1), respectively $O(1)$ in (2),
then all the conclusions of (1), respectively (2), hold 
for $\vphantom{\int^{X^A}}$\smash{$\frac{|\smallshazero_{A/K_n}[p^\infty]|}
  {|\smallshazero_{A'\!/K_n}[p^\infty]|}$} as well. 
\end{theorem}

\begin{remarks}
\label{remarks}
\noindent\par

\par\smallskip\noindent
{\bf(1)} If $p\nmid\deg\phi$, then the $p^\infty$-Selmer groups and 
the $p$-part of $\sha$ cannot have $\phi$-torsion, so the corresponding 
quotients in Theorems \ref{imain}-\ref{iLie} are trivial.
The results can be reduced to those for isogenies of $p$-power degree.

\par\smallskip\noindent
{\bf(2)}
In Theorems \ref{imain}, \ref{iZlmain},
the quotient \smash{$\frac{\Omega_{E'}}{\Omega_{E}}$} is a rational number
(Lemma~\ref{perrat}).
For $p$-iso\-genous curves over $\Q$ it is $1$ or $p^{\pm 1}$,
see \cite{isogloc} Thm 8.2.
%
The term $\frac1{12}{\ord_p(\frac{\Delta_{E'}}{\Delta_E})}$ is 0 unless $E$ 
has additive potentially supersingular reduction at $p$, 
see \cite{isogloc} Table~1.
In this exceptional case, $\mu$ does not have to be an integer, see 
Example~\ref{expotsup}.

\par\smallskip\noindent
{\bf(3)} Suppose $\Gal(K_\infty/K)\!=\!\Gamma\!\iso\!\Z_p$. If
the dual $p^\infty$-Selmer group of $E/K_\infty$ is a torsion 
$\Zp[[\Gamma]]$-module, then the invariant $\mu$ of 
Theorem \ref{iZlmain} is $\mu(E)\!-\!\mu(E')$, the difference of
the classical $\mu$-invariants of the two Selmer groups over~$K_\infty$.
In this setting, Theorem \ref{iZlmain} is equivalent to a theorem
of Schneider (for odd~$p$), see \cite{Sch, PeV}.

\par\smallskip\noindent
{\bf(4)} 
Suppose $E/\Q$ has good ordinary reduction at $p$. 
Then the dual $p^\infty$-Sel\-mer group of $E$ over the $p$-cyclotomic 
extension over $\Q$ is a torsion 
Iwasawa module, by Kato's theorem \cite{KatP}. 
A conjecture of Greenberg (\cite{Gre2}~Conj~1.11) asserts that
the isogeny class of $E$ contains a curve of $\mu$-invariant 0.
Granting the conjecture, Theorem \ref{imain} implies that
\begin{enumerate}
\item[(i)]
In the isogeny class of $E/\Q$ the curve $E_m$ with the largest period
$\Omega_{E_m}$ has $\mu$-invariant~0 at all primes of good ordinary
reduction.
\item[(ii)]
$E$ has $\mu$-invariant $\ord_p\frac{\Omega_{E_m}}{\Omega_E}$ at $p$.
\end{enumerate}
Thus, the theorem provides a conjectural formula for the $\mu$-invariant.
We note here that Greenberg's conjecture is known not to hold in general 
over number fields, see \cite{DriS}.
%



\par\smallskip\noindent
{\bf(5)} 
Greenberg (see \cite{Gre} Exc. 4.3--4.5) 
has observed that if $\phi: E\to E'$ is a $p$-isogeny over $K$ 
whose kernel is $\bbmu_p$, then the map
$$
  K^\times/K^{\times p} \iso H^1(\Gal(\bar K/K),\bbmu_p)\lar H^1(\Gal(\bar K/K),E[p])
$$
induced by the inclusion $\bbmu_p\subset E[p]$ gives a way to construct 
classes in the $p$-Selmer group of $E$. The units of $K$ contribute 
to the Selmer group, and the rank of the unit group is $\!\sim\![K\!:\!\Q]$.
In particular, one can exhibit \hbox{$\mu$-like} growth of $\Sel_p(E)$ 
in towers $K_n/K$. It would be interesting to 
similarly explain the Selmer growth
in Theorems \ref{imain}--\ref{iLie} for $p$-power isogenies 
with arbitrary kernels.

\par\smallskip\noindent
{\bf(6)} 
By a theorem of Washington \cite{Was}, for $p\ne l$ the $p$-part of the ideal
class group is bounded in the $l$-cyclotomic tower. 
Theorem \ref{iZlmain} provides examples of elliptic curves over $\Q$
for which the analogous statement for the Tate-Shafarevich group fails,
see e.g. Example \ref{11a3ord}.

In the opposite direction, Lamplugh \cite{Lam} has recently proven 
the following analogue of the theorem of Washington for elliptic curves $E/\Q$ 
with complex multiplication by the ring of integers of an imaginary quadratic 
field~$K$. Let $p>3$, $l>3$ be distinct primes of good reduction of $E$
that split in $K/\Q$. 
Lamplugh proves that if $K_n$ is the $n$th layer of the unique 
$\Z_l$-extension of $K$ unramified outside one of the factors of $l$ in $K$,
then the $p^\infty$-Selmer group 
of $E$ over $K_n$ stabilises as $n\to\infty$.
%


\par\smallskip\noindent
{\bf(7)} 
The constants $\mu$ and $\mu_1,...,\mu_d$ in Theorem \ref{iLie}
can be made explicit, as in Theorem \ref{iZlmain}.
Following the proof of Theorem \ref{Zlmain}, this requires the knowledge
of the decomposition and inertia groups at bad primes; the other ingredients 
are computed in Proposition \ref{gammamain}.
\end{remarks}

\begin{example}[Ordinary reduction]
\label{11a3ord}
Let us show that the curves $11A1, 11A2$ 
have unbounded 5-primary part of $\sha$ in the
cyclotomic $\Z_l$-extension of $\Q$ for every prime $l$. 
There are 5-isogenies
$$
  11A2 \lar 11A1 \lar 11A3,
$$
and $\Omega_{11A3}=5\Omega_{11A1}=25\Omega_{11A2}=6.34604...$. 
So, by Theorem \ref{imain}, for every~$l$ there are $\nu_l, \nu'_l\in\Z$ 
such that
$$
  |\sha_{11A2/\Q(l^n))}[5^\infty]| \ge 25^{l^n - \nu_l}, \qquad
  |\sha_{11A1/\Q(l^n))}[5^\infty]| \ge 5^{l^n - \nu'_l}.
$$

A standard computation with cyclotomic Euler characteristics 
(e.g. as in \cite{iwacomp} \S3.11) shows that 
for every ordinary prime $l$ for which $a_l\ne 1$, 
e.g. \hbox{$l=3,7,13,17,...$}, the curves have rank 0 over $\Q(l^n)$, for all $n\ge 1$.
For such primes $\nu_l, \nu'_l$ can be taken to be $1$ and the number 
of primes above 11 in $\bigcup_n\Q(l^n)$, respectively\footnote
{So $\nu'_l$ is almost always 1 as well; for $l<10^7$ the only exception is $l=71$.}. 
For $l=5$, these bounds are exact, as $\sha_{11A3/\Q(5^n)}$ 
is known to have trivial 5-primary part for all $n\ge 1$.
%
\end{example}

\begin{example}[Potentially supersingular reduction]
\label{expotsup}
Let $E/\Q$ be an elliptic curve with good supersingular reduction at $p$,
and $K_n=\Q(p^n)$, the $n$th layer in the $p$-cyclotomic tower.
By a theorem of Kurihara \cite{Kur}, under suitable hypothesis,
$$
  |\sha_{E/K_n}[p^\infty]| = p^{\lfloor\mu p^n - \frac 12\rfloor}, \qquad
  \mu = \frac{p}{p^2-1}.
$$
Note that such curves cannot have a $p$-isogeny, by a theorem of 
Serre\linebreak(\cite{SerP}~Prop. 12). In contrast, elliptic curves over $\Q$ with
additive potentially supersingular reduction at $p$ can have a $p$-isogeny,
and there are examples for which 
$$
  |\sha_{E/K_n}[p^\infty]| \ge p^{\,\mu p^n + \nu} 
$$
with $\mu>1$. For instance, there is a 9-isogeny $\phi: 54A2\to54A3$.
These curves have potentially supersingular reduction at $p=3$, and
$$
  \Omega_{54A3}=9\,\Omega_{54A2}, \quad \Delta_{54A2}=-2^93^{11}, \quad 
  \Delta_{54A3}=-2\cdot 3^3.
$$
By Theorem \ref{imain}, there is a constant $\nu$ such that for all 
large enough $n$,
$$
  |\sha_{54A2/K_n}[3^\infty]| \ge  
  \frac{|\sha_{54A2/K_n}[3^\infty]|}{|\sha_{54A3/K_n}[3^\infty]|} =  
  3^{\,3^n\!\mu+\nu}, \quad\>\> \mu 
    = 2-\frac{11-3}{12} = \frac 43.
$$
\end{example}

\begin{example}[False Tate curve tower]
To illustrate Theorem \ref{iLie} for a higher-dimensional $l$-adic Lie group,
let $K_n=\Q(\zeta_{3^n}, \sqrt[3^n]{7})$, a `false Tate curve tower' 
in the terminology of \cite{HV, iwacomp}. 
Let $E=11A1$ and $E'=11A3$, as in Example \ref{11a3ord}. 
We find (see Example \ref{falseTate}) 
that either $\sha_{E/K_n}[5^\infty]$ is infinite or 
$$
  \frac{|\sha_{E/K_n}[5^\infty]|}{|\sha_{E'\!/K_n}[5^\infty]|} = 5^{3^{2n-1}-3^n}.
$$
\end{example}


\begin{example}
Let $K_\infty$ be the unique $\Z_5^2$-extension of $\Q(i)$
and let $K_n$ be its $n$th layer;
thus $\Gal(K_n/\Q)\iso C_{5^n}\times D_{2\cdot 5^n}$.
If we take the 5-isogenous curves $E=75A1, E'=75A2$ over $\Q$, 
with additive potentially supersingular reduction at 5,
we find that (see Example \ref{ex75a1})
$$
\frac{|\Sel^{\div}_{5^\infty}(E/K_n)[\phi]|}
     {|\Sel^{\div}_{5^\infty}(E'\!/K_n)[\phi^t]|} 
\frac{|\shazero_{E/K_n}[5^\infty]|}{|\shazero_{E'\!/K_n}[5^\infty]|} 
  = 
  5^{\,\mu 5^{2n} + \mu_1(n) 5^n + \mu_2(n)}
$$
with 
$$
  \mu=-\frac13, \qquad
  \mu_1(n) = 1-\frac23(-1)^n, \qquad
  \mu_2(n) = 0.
$$
So the assumption that $E$ does not have potentially supersingular
reduction in Theorem \ref{iLie}(2) cannot be removed,
as the $\mu_i$ may fluctuate with $n$.
\end{example}


\begin{example}
As opposed to the cyclotomic extensions, 
for general\linebreak $\Z_l$-extensions of number fields 
there is an extra term in $\mu$ coming from Tamagawa numbers 
(compare Theorem \ref{iZlmain} with Theorem \ref{imain}).
For example, consider the 5-isogeny $11A1\to 11A3$ as in Example \ref{11a3ord}
in the 5-anticyclotomic tower $K_\infty$ of $K=\Q(i)$. 
Because $11$ is inert in $\Q(i)$, and so totally split in $K_\infty/K$, 
there is a $\mu$-contribution from the Tamagawa numbers
(5 and 1) in this $\Z_5$-extension, but not in the cyclotomic one.
\end{example}


\begin{remark}[CM curves with $\mu>0$]
If $K_\infty/K$ is a $\Z_p$-extension and $\Sel_{p^\infty}(E/K_\infty)$ is
cotorsion over the Iwasawa algebra of $\Gal(K_\infty/K)$, then it has a 
well-defined $\mu$-invariant as in classical Iwasawa theory. 
Theorem~\ref{iZlmain} gives a formula for its change under isogenies in
terms of elementary invariants, and allows us to generate examples with
positive $\mu$-invariant. 

Consider, for instance, elliptic curves with complex multiplication
and good ordinary reduction at $p$. 
Such examples over $\Q$ with a $p$-isogeny
are almost non-existent: there are 13 CM $j$-invariants over $\Q$, 
and there is only one with a $p$-isogeny that admits good reduction at $p$.
It is $j=-3^35^3$ (CM by $\Z[{\scriptscriptstyle\frac{1+\sqrt{-7}}2}]$), 2-isogenous
to $j=3^35^317^3$ (CM~by $\Z[\sqrt{-7}]$). 
(This is easy to check from the table
of CM $j$-invariants \cite{Sil2} Appendix A and by computing the isogenous curves,
e.g. in Magma~\cite{Magma}.) The simplest example with these $j$-invariants is 
$$
  \phi: 49A1 \lar 49A2. 
$$
Here $\Omega_{49A1}/\Omega_{49A2}=2$, and so $49A2$ does have 
positive $\mu$-invariant for $p=2$
\par\pagebreak\noindent
as well as unbounded 2-part of $\sha$ in every cyclotomic $\Z_l$-extension of $\Q$,
by Theorem~\ref{imain}.
Assuming Greenberg's conjecture (Remark \ref{remarks}(4)), 
the curve $49A2$ and $p\!=\!2$ is the unique 
example (up to quadratic twists) of a good ordinary CM 
curve over $\Q$ with positive $\mu$-invariant.

Over larger number fields, other examples are easy to construct.
For instance, the curve
$$
  E: y^2   =   x^3   -   24z^7\sqrt{z+3}\,x^2   +   zx,\quad\qquad z={\frac{\sqrt5-1}2}
$$
is defined over $K=\Q(\zeta_{20})^+=\Q(\sqrt{z+3})$  
and has CM by $\Z+5i\Z$. It has good ordinary reduction at the prime above $5$,
and is 5-isogenous to $y^2=x^3+zx$. 
Computing the periods and applying Theorem \ref{iZlmain}, we find that
it should have positive $\mu$-invariant both over the
$\Z_5$-cyclotomic extension of~$K$, and over every $\Z_5$-extension of
$K(i)=\Q(\zeta_{20})$. 

Let $F_\infty$ be the composite of all $\Z_5$-extensions of $K(i)$, so that 
$G=\Gal(F_\infty/K)\iso\Z_5^5$.
The $5^\infty$-Selmer group of $E$ over
$F_\infty$ is conjectured to satisfy the $\cM_H(G)$-conjecture 
of non-commutative Iwasawa theory \cite{CFKSV}. 
As John Coates remarked to us, this example provides evidence
for the conjecture as follows. Similar arguments to those given in \cite{CS} 
would show that the $\cM_H(G)$-conjecture implies that the $G$-$\mu$-invariant 
of the Selmer group over $F_\infty$ 
would have to be equal to the usual $\mu$-invariant 
of the Selmer group over the cyclotomic $\Z_5$-extension of $K(i)$,
which we have shown to be non-zero. Thus, granted the $\cM_H(G)$-conjecture, 
it would follow that the $G$-$\mu$-invariant of the Selmer group over $F_\infty$ 
would have to be positive, and then an easy further argument shows that 
the $\mu$-invariant over every $\Z_5$-extension of $K(i)$ 
would also be positive, in accord with what we have proven.
%
\end{remark}

\bigskip

\subsection*{Brief overview of the paper}

To control the change of Selmer groups, we invoke the theorem
on the invariance of the Birch--Swinnerton-Dyer conjec\-ture under isogeny
by Cassels and Tate \cite{CasVIII, TatC}. 
This is recalled in \hbox{Theorem}~\ref{thmsel} in \S\ref{sinv},
after we introduce a convenient choice of periods in~\S\ref{sper}. 
Pretty much the rest of the paper studies how the terms of the 
Birch--Swinnerton-Dyer formula behave in towers of local fields 
and number fields: minimal differentials (\S\ref{som}), 
Tamagawa numbers (\S\ref{stam}), torsion (\S\ref{stors}) and the divisible
part of Selmer (\S\ref{sdiv}).
At the end of \S\ref{sinv} we also give some examples how this procedure
works. 
Theorems \ref{imain}-\ref{iLie} are proved in \S\ref{sgrowth}.
In \S\ref{som}, \S\ref{stam} we rely on the results of \cite{isogloc} 
that describe how local invariants of elliptic curves change under isogeny.

The appendix (\S\ref{sapp}) concerns the behaviour of conductors of elliptic curves
and Galois representations in field extensions. There is no assumption on
the existence of an isogeny here, and the results may be of independent
interest.

\newpage

\subsection*{Notation}
\label{not}
We write $E, E'$ for elliptic curves and $A, A'$ for abelian varieties. 
We usually have an isogeny $E\to E'$ or $A\to A'$, denoted by $\phi$. 
Its dual $E'\to E$ or $(A')^t\to A^t$ is denoted by $\phi^t$. 
Number fields are denoted by $K, F, ...,$ and $l$-adic fields 
(finite extensions of $\Q_l$) by $\K, \F, ...$. 
We also use the following notation:

\bigskip

\begin{tabular}{llll}
\vphantom{$\int^X$}%
$|\cdot|_v$ & normalised absolute value at a prime $v$\cr
$v(\cdot)$, $\ord_v(\cdot)$ & normalised valuation in a local field/at $v$\cr
$j_E$ & $j$-invariant of an elliptic curve $E$ \cr
\end{tabular}

\begin{tabular}{llll}
$\Delta_E, \Delta_{E/K}$ & minimal discriminant of an elliptic curve over $K$\cr
$\delta$, $\delta'$ & $v(\Delta_{E/\K}), v(\Delta_{E'\!/\K})$ when $\K$ is local\cr
$f_{E/\K}$ & conductor exponent of $E/\K$ when $\K$ is local\cr
\vphantom{$\int^X$}%
$\Om_\phi(\F)$ & see Definition \ref{defom} \cr
$\omega_v^{\min}=\smash{\omega_{A,v}^{\min}}$ & N\'eron minimal exterior form of an abelian variety at a \cr
  & prime $v$ (minimal differential for an elliptic curve)\cr
$c_{A/K}, c_v(A/K)$ & Tamagawa number over a local field/global field at $v$\cr
$\Omega, \Omega^\ast$ & infinite periods, see Definition \ref{defOmega}\cr
$\sha, \sha^{\div}, \shazero$ & Tate-Shafarevich group, its divisible part,
  $\shazero\,=\!\sha/\sha^{\div}$ \cr
$\Sel_{p^\infty}(A/K)$ & $\varinjlim\Sel_{p^n}(A/K)$, the $p$-infinity Selmer group\cr
$\rk_p(A/K)$ & $\Zp$-corank of $\Sel_{p^\infty}(A/K)$ \cr
$\Q(l^n)$ & the $n$th layer of the cyclotomic $\Z_l$-extension of $\Q$, i.e.\cr
  & the unique totally real degree $l^n$ subfield of $\Q(\zeta_{l^{n+2}})$\cr 
$X[\phi], X[p^\infty]$  & $\phi$-torsion, $p$-primary component 
  of an abelian group $X$\cr
$\lfloor \cdot\rfloor, \{\cdot\}$ & integer part (floor) and fractional part 
  $\{x\}\!=\!x\!-\!\lfloor x\rfloor$\cr
\end{tabular}

\bigskip

Any two non-zero invariant exterior forms $\omega_1, \omega_2$ on an abelian variety $A/K$ 
are multiples of one another, $\omega_1=\alpha\omega_2$ with $\alpha\in K$.
We will write $\omega_1/\omega_2$ for the scaling factor $\alpha$.

When $\K$ is an $l$-adic field, recall that an elliptic curve $E/\K$ has 
additive reduction if and only if it has conductor exponent $f_{E/\K}\ge 2$,
and that $f_{E/\K}=2$ if and only if the $\ell$-adic Tate module of $E$ is tamely ramified
for some (any) $\ell\ne l$.
We will call this {\em tame} reduction, and {\em wild} if $f_{E/\K}>2$.
If $l\ge 5$, the reduction is always tame.
We remind the reader that $E/\K$ has potentially good reduction if
$v(j_E)\ge 0$ and potentially multiplicative reduction if $v(j_E)<0$.

Finally, an \emph{$l$-adic Lie group} $G$ is a closed subgroup $\GL_k(\Z_l)$
for some $k$; it has a natural filtration by open subgroups, the kernels of 
the reduction maps mod $l^n$.

We use Cremona's notation (such as $11A1$) for elliptic curves over $\Q$.

\smallskip

\begin{acknowledgements}
This investigation was inspired by conversations with John Coates. 
We would like to thank him and also Masato Kurihara, Ralph Greenberg 
and Karl Rubin for their comments. 
The first author is supported by a Royal Society
University Research Fellowship.
\end{acknowledgements}

\newpage

\section{Periods}
\label{sper}

We introduce a convenient form of periods of abelian varieties over
number fields, that are model-independent and well-suited for the 
Birch--Swinnerton-Dyer conjecture.

\begin{definition}
\label{defOmega}
An abelian variety $A/K$ has a non-zero invariant exterior form $\omega$, 
unique up to $K$-multiples. If $K=\C$, we define the \emph{local period}
$$
  \Omega_{A/\C,\omega} = \int_{A(\C)}\!2^{\dim A} \omega\!\wedge\omegabar.
$$
If $K=\R$, define
$$
  \Omega_{A/\R,\omega} = \int_{A(\R)} |\omega|
    \qquad\text{and}\qquad
  \Omega_{A/\R}^\ast = \frac{\Omega_{A/\C,\omega}}{\Omega_{A/\R,\omega}^2}.
$$
If $K$ is a number field, define the \emph{global period}
$$
  \Omega_{A/K} \>=\>  \prod_{v\nmid\infty} |\omega/\omega_{v}^{\min}|_v \,
    \prod_{v|\infty} \Omega_{A/K_v,\omega}.
$$
\vskip-1mm\par\noindent
Here $v$ runs through places of $K$, and the term at $v$ in the first product is
the normalised $v$-adic absolute value of the quotient 
of $\omega$ by the N\'eron minimal form at $v$. 
\end{definition}

%

\begin{remark}
\label{perQ}
Note that both $\Omega_{A/K_v}^\ast$ and $\Omega_{A/K}$
are independent of the choice of $\omega$, by the product formula 
for the second one.

An elliptic curve $E/\Q$ can be put in minimal Weierstrass form, 
with the global minimal differential $\omega=\frac{dx}{2y+a_1x+a_3}$.
Then $\Omega_{E/\Q}=\Omega_{E/\R,\omega}$, which is the traditional
real period $\Omega_+$ or $2\Omega_+$, depending on whether or not 
$E(\R)$ is connected. 
If $E(\C)=\C/\Z\tau\!+\!\Z$ under the usual complex uniformisation,
then $\Omega^\ast_{E/\R}=\Im\tau$. 
\end{remark}

\begin{lemma}
\label{perrat}
Let $\phi: A\to A'$ be an isogeny of abelian varieties over a number field $K$. 
Then both $\frac{\Omega_{A'\!/K}}{\Omega_{A/K}}$ and, for real places $v$, 
$\frac{\Omega_{A'\!/K_v}^\ast}{\Omega_{A/K_v}^\ast}$
are positive rational numbers. They have trivial $p$-adic valuation for all 
$p\nmid\deg\phi$.
\end{lemma}

\begin{proof}
Fix a non-zero invariant exterior form $\omega'$ on $A'$, and 
set $\omega=\phi^*\omega'$. 
For $v|\infty$, 
$$
  \frac{\Omega_{A'\!/K_v,\omega'}}{\Omega_{A/K_v,\omega}} =  
  \frac{|\coker\phi: A(K_v)\to A'(K_v)|}{|\ker\phi: A(K_v)\to A'(K_v)|}.
$$
This is a positive rational, and considering the conjugate isogeny 
$\phi': A'\to A$ (so that $\phi'\circ\phi, \phi\circ\phi'$ are the 
multiplication-by-$\deg\phi$ maps), we see that 
the only prime factors of $|\ker|$ and $|\coker|$ are those dividing 
$\deg\phi$. The claim for $\Omega^\ast$ now follows.

As for the global periods, 
$$
  \frac{\Omega_{A'\!/K}}{\Omega_{A/K}}
  =
  \prod_{v\nmid\infty} \frac{|\omega'\!/\omega_{A',v}^{\min}|_v}{|\omega/\omega_{A,v}^{\min}|_v}
  \prod_{v|\infty}\frac{\Omega_{A'\!/K_v,\omega'}}{\Omega_{A/K_v,\omega}}
$$
is a positive rational. If $v\nmid\deg\phi$, then 
$\phi^*(\omega_{A',v}^{\min})$ is a unit multiple of 
$\omega_{A,v}^{\min}$, so 
$\frac{|\omega'\!/\omega_{A',v}^{\min}|^{}_v}{|\omega/\omega_{A,v}^{\min}|^{}_v}=1$. 
So $\frac{\Omega_{A'\!/K}}{\Omega_{A/K}}$
has trivial $p$-adic valuation at primes $p\nmid\deg\phi$.
\end{proof}

\begin{lemma}
\label{lemper}
Let $\phi: A\to A'$ be an isogeny of abelian varieties over a number field $K$, 
and $F/K$ a finite extension. Then
$$
  \Omega_{A/F}
      = 
  \Omega_{A/K}^{[F:K]}
  \prod_{\text{$v$ real}}\bigl({\Omega^\ast_{A/K_v}}\bigr)^{\#\{w|v\text{ complex}\}}
  \prod_{v,\>w|v}\>\,
  \biggl|
  \frac{{\omega_{v}^{\min}}}{\omega_{w}^{\min}}
  \biggr|_{w},
$$
where $v$ runs over places of $K$ and $w$ over places of $F$ above $v$.
\end{lemma}

\begin{proof}
Choose an invariant exterior form $\omega$ for $A/K$. We compute the terms in 
$\Omega_{A/F}$ using $\omega$. 

Let $v$ be a place of $K$. 
If $v$ is complex, then 
$
  \prod_{w|v} \Omega_{A/F_w,\omega} = \Omega_{A/K_v,\omega}^{[F:K]}.
$
If $v$ is real, then, writing $\Sigma_+, \Sigma_-$ for the set of real and
complex places $w|v$ in~$F$, we have
$$
  \prod_{w|v} \Omega_{A/F_w,\omega} = 
  \prod_{w\in\Sigma_+} \Omega_{A/K_v,\omega} 
    \prod_{w\in\Sigma_-} \Omega_{A/K_v,\omega}^2 \Omega^\ast_{A/K_v} 
      = \Omega_{A/K_v,\omega}^{[F:K]} (\Omega^\ast_{A/K_v})^{|\Sigma_-|}.
$$
If $v\nmid\infty$, then
$$
  \prod_{w|v}\,
    \Bigl|
  \frac{{\omega}}{\omega_{w}^{\min}}
  \Bigr|_{w} = 
  \prod_{w|v}\, 
    \Bigl|
  \frac{{\omega}}{\omega_{v}^{\min}}
  \Bigr|_{w}
    \Bigl|
  \frac{{\omega_v^{\min}}}{\omega_{w}^{\min}}
  \Bigr|_{w} =  
    \Bigl|
  \frac{{\omega}}{\omega_{v}^{\min}}
  \Bigr|_{v}^{[F:K]}
  \prod_{w|v}\,
  \Bigl|
  \frac{{\omega_v^{\min}}}{\omega_{w}^{\min}}
  \Bigr|_{w}.
$$
Multiplying the terms over all places $v$ of $K$ gives the claim. 
\end{proof}

\begin{remark}
\label{remdiffdisc}
For elliptic curves, the term
${\omega_{v}^{\min}/\omega_{w}^{\min}}$ relates to the behaviour of the 
minimal discriminant of $E$ in $F_w/K_v$ (cf. \cite{Sil1} Table III.1.2),
$$
  \ord_w\Bigl(\frac{\omega_{w}^{\min}}{\omega_{v}^{\min}}\Bigr)=\frac{1}{12}\>
  \ord_w\Bigl(\frac{\Delta_{E/K}}{\Delta_{E/F}}\Bigr).
$$
\end{remark}

\section{BSD invariance under isogeny}
\label{sinv}

We now state a version of the invariance of the Birch--Swinnerton-Dyer
conjecture under isogeny for Selmer groups (see page \pageref{not} for 
the notation).

%
%
%

\begin{theorem}
\label{thmsel}
Let $\phi: A\to A'$ be a isogeny of abelian varieties over a number field $K$,
and $\phi^t: A'^t\to A^t$ the dual isogeny.
If the degree of $\phi$ is a power of $p$, then
\begingroup\smaller[1]
$$
\frac{
|\Sel^{\div}_{p^\infty}(A/K)[\phi]|
}{
|\Sel^{\div}_{p^\infty}({A'}^t/K)[\phi^t]|
} 
  \frac{|\shazero_{A/K}[p^\infty]|}{|\shazero_{A'\!/K}[p^\infty]|}
=  
  \frac{|A(K)[p^\infty]||A^t(K)[p^\infty]|}{|A'(K)[p^\infty]||A'^t(K)[p^\infty]|}
  \frac{\Omega_{A'\!/K}}{\Omega_{A/K}}
  \prod_{v\nmid\infty}
  \frac{c_v(A'\!/K)}{c_v(A/K)}.
$$
\endgroup
Otherwise, the left-hand side and the right-hand side have the same $p$-part.
\end{theorem}

\begin{proof} 
This is essentially \cite{squarity} Thm 4.3, that says
\begingroup\smaller[1]
$$
\frac{
|\Sel^{\div}_{p^\infty}(A/K)[\phi]|
}{
|\Sel^{\div}_{p^\infty}({A'}^t/K)[\phi^t]|
} 
  \prod_{\scriptscriptstyle p|\deg\phi}\!\frac{|\shazero_{A/K}[p^\infty]|}{|\shazero_{A'\!/K}[p^\infty]|}
=  
  \frac{|A(K)_{\scriptscriptstyle\tors}||A^t(K)_{\scriptscriptstyle\tors}|}{|A'(K)_{\scriptscriptstyle\tors}||A'^t(K)_{\scriptscriptstyle\tors}|}
  \frac{\Omega_{A'\!/K}}{\Omega_{A/K}}
  \prod_{\scriptscriptstyle v\nmid\infty}
  \frac{c_v(A'\!/K)}{c_v(A/K)}.
$$
\endgroup
The term $Q(\phi)$ in \cite{squarity} in exactly $\Sel^{\div}_{p^\infty}(A/K)[\phi]$.
\end{proof}

\newpage

\begin{corollary}
\label{selmain}
Let $\phi: A\to A'$ be an isogeny of abelian varieties over a number field $K$
with dual $\phi^t: A'^t\to A^t$, and $F/K$ a finite extension. 
If the degree of $\phi$ is a power of~$p$, then
$$
\frac{
|\Sel^{\div}_{p^\infty}(A/F)[\phi]|
}{
|\Sel^{\div}_{p^\infty}(A'^t/F)[\phi^t]|
} 
  \frac{|\shazero_{A/F}[p^\infty]|}{|\shazero_{A'\!/F}[p^\infty]|} =
  \frac{|A(F)[p^\infty]||A^t(F)[p^\infty]|}{|A'(F)[p^\infty]||A'^t(F)[p^\infty]|}
  \Bigl(\frac{\Omega_{A'\!/K}}{\Omega_{A/K}}\Bigr)^{[F:K]} \times
$$
$$
  \times
  \prod_{\text{$v$ real}}
  \bigl(\frac{{\Omega^\ast_{A'\!/K_v}}}{{\Omega^\ast_{A/K_v}}}\bigr)^%
  {\#\{w|v\text{ complex}\}} 
  \prod_{v\nmid\infty}
  \frac{c_v(A'\!/F)}{c_v(A/F)}
  \>\,
  \prod_{v,\>w|v}\>\,
  \biggl|
  \frac
  {{{\omega_{A',v}^{\min}}}/{\omega_{A',w}^{\min}}}
  {{{\omega_{A,v}^{\min}}}/{\omega_{A,w}^{\min}}}
  \biggr|_{w},
$$
where $v$ ranges over places of $K$, and $w|v$ are places of $F$.
If $\phi$ has arbitrary degree, then
the left-hand side and the right-hand side have the same $p$-part.
\end{corollary}

\begin{proof}
Combine Theorem \ref{thmsel} with Lemma \ref{lemper}.
\end{proof}

Corollary \ref{selmain} is our main tool for studying the Selmer growth
in towers in~\S\ref{sgrowth}. 
As we now illustrate, it already enables us to construct explicit examples
of interesting growth of Selmer and $\sha$. 
The general behaviour of the Tama\-gawa number quotient 
will be discussed in \S\ref{stam}, 
torsion quotient in \S\ref{stors},
and 
the contribution from exterior forms in \S\ref{som},
under the name of $\Om_\phi(F_w)$.

%


\begin{example}
\label{falseTate}
Let $K_n=\Q(\zeta_{3^n}, \sqrt[3^n]{7})$, a `false Tate curve tower' 
in the\linebreak terminology of \cite{HV, iwacomp}, 
and let $\phi: E=11A1\to E'=11A3$ be the 5-isogeny\linebreak
as in Example \ref{11a3ord}. 
A result of Hachimori and Matsuno \cite{HM} Thm 3.1
and a cyclotomic Euler characteristic computation as in \cite{iwacomp} \S3.11
show that $\rk E/K_n\!=\!\rk_3 E/K_n\!=\!0$ for all $n\ge 1$. Therefore
$$
  \Sel_{5^\infty}(E/K_n) \quad=\quad \sha_{E/K_n}[5^\infty],
$$
and similarly for $E'$. 

The periods of the two curves are 
$$
  \Omega_{E/\Q}=1.2692...=\frac15\Omega_{E'\!/\Q}, \qquad
  \Omega^*_{E/\R}=1.1493...=5\Omega^*_{E'\!/\R},
$$
and both curves have torsion of size 5 over all $K_n$. 
Applying Corollary \ref{selmain}, we find that either 
$\sha_{E/K_n}[5^\infty]$ is infinite for some $n$, or
\beq
  \displaystyle
  \frac{|\sha_{E/K_n}[5^\infty]|}{|\sha_{E'/K_n}[5^\infty]|} &=&
    \displaystyle
    \frac{5^2}{5^2}\cdot
    \bigl(\frac{\Omega_{E'\!/\Q}}{\Omega_{E/\Q}}\bigr)^{2\cdot 3^{2n-1}}\cdot
    \bigl(\frac{\Omega^*_{E'\!/\R}}{\Omega^*_{E/\R}}\bigr)^{3^{2n-1}}
    \cdot \prod_{v|11} \frac{c_v(E'\!/K_n)}{c_v(E/K_n)}\cdot 1\cr
  &=& \displaystyle \frac{5^{2\cdot 3^{2n-1}}}{5^{3^{2n-1}}\cdot 5^{3^n}} =
      5^{3^{2n-1}-3^n}.
\eeq

\end{example}

\begin{example}[Fluctuation in Selmer growth]
\label{ex75a1}
In the previous example, the quotient of the Tate-Shafarevich groups grew like
$5^{3^{2n-1}-3^n}$. Theorem~\ref{iLie} shows that such growth of the form
$p^{\text{polynomial in $l^n$}}$ is a general phenomenon. However,
the assumption on primes of additive potentially supersingular
reduction is essential, as we now illustrate.


Let $K=\Q$ and let $K_\infty$ be the unique $\Z_5^2$-extension of $\Q(i)$, so 
that
$$
  \Gal(K_\infty/\Q)\>\>\iso\>\>\Z_5\times(\Z_5\rtimes C_2).
$$
Write $K_n$ for the $n$th
layer of $K_\infty/K$, so $\Gal(K_n/\Q)\iso C_{5^n}\times D_{2\cdot 5^n}$.

Consider the following curves over $\Q$, that are connected by a 5-isogeny 
$\phi: E\to E'$,
\beq
  E:  &y^2 + y = x^3 - x^2 - 8x - 7    \qquad& (75A1,\> \Delta_{E/\Q}=-3\cdot 5^4),\cr
  E': &y^2 + y = x^3 - x^2 + 42x + 443 \qquad& (75A2,\> \Delta_{E'\!/\Q}=-3^5\cdot 5^8). 
\eeq
They have non-split multiplicative reduction at $p=3$, of Kodaira type
$\In{1}$ and $\In{5}$, respectively, and 
additive potentially supersingular reduction at $p=5$, of Kodaira 
type~\IV{} and \IVS, respectively. Their periods are
$$
  \Omega_{E/\Q}=1.4025...=\Omega_{E'\!/\Q}, \qquad
  \Omega^*_{E/\R}=1.6646...=5\Omega^*_{E'\!/\R}.
$$
Both curves have trivial torsion over $\Q(i)$, and therefore no 5-torsion
over $K_n$, by Nakayama's lemma. 

By Corollary \ref{selmain},
$$
\frac{|\Sel^{\div}_{5^\infty}(E/K_n)[\phi]|}
     {|\Sel^{\div}_{5^\infty}(E'\!/K_n)[\phi^t]|} 
\frac{|\shazero_{E/K_n}[5^\infty]|}{|\shazero_{E'\!/K_n}[5^\infty]|} 
   =
  \Bigl(\frac{\Omega_{E'\!/K}}{\Omega_{E/K}}\Bigr)^{2\cdot5^{2n}}
  \Bigl(\frac{{\Omega^\ast_{E'\!/\R}}}{{\Omega^\ast_{E/\R}}}\Bigr)^{5^{2n}} 
     \>\times \qquad \qquad \qquad \qquad \qquad
$$
$$
  \qquad \qquad \qquad \qquad \times\,
  \prod_{v|3}
  \frac{c_v(E'\!/K_n)}{c_v(E/K_n)}
  \prod_{v|5}
  \frac{c_v(E'\!/K_n)}{c_v(E/K_n)}
  \>\,
  \prod_{v|5}\>\,
  \biggl|
  \frac
  {{{\omega_{E',5}^{\min}}}/{\omega_{E',v}^{\min}}}
  {{{\omega_{E,5}^{\min}}}/{\omega_{E,v}^{\min}}}
  \biggr|_{v}.
$$
%
The prime $p=3$ is inert in $\Q(i)$ and in the 5-cyclotomic tower,
and the prime above it in $\Q(i)$ is totally split in the 5-anticyclotomic tower
of $\Q(i)$. So there are $5^n$ primes above $3$ in $K_n$. 
The curves $E$ and $E'$ have split multiplicative reduction at each of them
(of type $\In{1}$ and $\In{5}$), and so
$$
  \prod_{v|3}
  \frac{c_v(E'\!/K_n)}{c_v(E/K_n)} = 5^{5^n}.
$$
The Tamagawa numbers at $v|5$ in $K_n$ are coprime to 5 (potentially good reduction),
but they do contribute to the quotient of $\omega$'s. Specifically, 
there are two primes $v_n^+, v_n^-$ above $5$ in $K_n$ (one above $2+i$ and one above $2-i$ 
in $\Q(i)$), both with residue degree $5^n$ and ramification degree $5^n$.

Remark \ref{remdiffdisc} lets us the compute the $\omega$-term. 
For $v=v_n^\pm$,
$$
  \ord_v\Bigl(\frac{\omega_{E,v}^{\min}}{\omega_{E,5}^{\min}}\Bigr)=\frac{1}{12}\>
  \ord_v\Bigl(\frac{\Delta_{E/\Q}}{\Delta_{E/K_n}}\Bigr).
$$
The valuation $\ord_v(\Delta_{E/K_n})
\in\{0,1,...,11\}$ is
uniquely determined by the congruence  
$$
  \ord_v(\Delta_{E/K_n}) \equiv \ord_v(\Delta_{E/\Q}) \equiv \ord_v(5^4) \mod 12.
$$
Therefore,
$$
  \frac{1}{12}\>
  \ord_v\Bigl(\frac{\Delta_{E/\Q}}{\Delta_{E/K_n}}\Bigr) = \lfloor \frac{4\cdot 5^n}{12} \rfloor.
$$
Similarly, the corresponding term for $E'$ is $\lfloor \frac{8\cdot 5^n}{12} \rfloor$,
and we find that
$$
  \biggl|
  \frac
  {{{\omega_{E',5}^{\min}}}/{\omega_{E',v}^{\min}}}
  {{{\omega_{E,5}^{\min}}}/{\omega_{E,v}^{\min}}}
  \biggr|_{v} = 
  (5^{5^n})^
  {\lfloor \frac{8\cdot 5^n}{12} \rfloor - \lfloor \frac{4\cdot 5^n}{12} \rfloor}
   = (5^{5^n})^{\frac 13 5^n - \frac 13 (-1)^n}
   = 5^{\frac 13 5^{2n} - \frac 13 (-1)^n 5^n}.
$$
Combining everything together, we deduce that
\beq
  \displaystyle
\frac{|\Sel^{\div}_{5^\infty}(E/K_n)[\phi]|}
     {|\Sel^{\div}_{5^\infty}(E'\!/K_n)[\phi^t]|} 
\frac{|\shazero_{E/K_n}[5^\infty]|}{|\shazero_{E'\!/K_n}[5^\infty]|} 
  &=&\displaystyle
  {{1^{2\cdot 5^{2n}} \cdot 5^{-5^{2n}}}\cdot  5^{5^{n}}} \cdot
  [5^{\frac 13 5^{2n} - \frac 13 (-1)^n 5^n}]^2 \cr
    &=& \displaystyle 5^{-\frac13\,5^{2n} + (1-\frac23(-1)^n)\cdot 5^n}.
\eeq
\end{example}

\section{Minimal differentials}
\label{som}

In this section we investigate the behaviour of the last term 
in Corollary~\ref{selmain} (contribution from the exterior forms)
in towers of local fields. We study its valuation, denoted $\Om_\phi(\F)$ below,
and how it changes with $\F$.

\begin{definition}
\label{defom}
Let $\F/\K$ be a finite extension of $l$-adic fields, and write $v_\F, v_\K$ 
for their valuations. 
For an isogeny $\phi: A\to A'$ of abelian varieties over $\K$, define
$$
  \Om_\phi(\F) = 
    v_\F\Bigl(\frac{\omega_{A'\!/\F}^{\min}}{\omega_{A'\!/\K}^{\min}}\Bigr) - 
    v_\F\Bigl(\frac{\omega_{A/\F}^{\min}}{\omega_{A/\K}^{\min}}\Bigr), 
$$
so that
$$
  |k_\F|^{\Om_\phi(\F)} = 
  \biggl|
  \frac
  {{{\omega_{A'/\K}^{\min}}}/{\omega_{A'/\F}^{\min}}}
  {{{\omega_{A/\K}^{\min}}}/{\omega_{A/\F}^{\min}}}
  \biggr|_{\F},
$$
where $k_\F$ is the residue field of $\F$ and $|\cdot|_\F$ the normalised 
absolute value.
\end{definition}

\begin{lemma}
\label{semom0}
\noindent\par\noindent
\begin{enumerate}
\item[\llap{(1)\ \ }]{}\hskip-8pt 
If $A, A'\!/\K$ are semistable, then $\Om_\phi(\F)=0$ for every $\F/\K$.
\item[\llap{(2)\ \ }]{}\hskip-8pt 
If $\F'\!/\F/\K$ are finite and $\F'\!/\F$ is unramified, then 
$\,\Om_\phi(\F)\!=\!\Om_\phi(\F')$.
\end{enumerate}
\end{lemma}

\begin{proof}
(1) The minimal models of $A, A'$ and the minimal exterior forms over $\K$
stay minimal over $\F$. (2) Ditto for $\F'\!/\F$.
\end{proof}

We now restrict our attention to elliptic curves. 
The term $v_\F\Bigl(\frac{\omega_{E/\F}^{\min}}{\omega_{E/\K}^{\min}}\Bigr)$ 
measures the extent to which the minimal Weierstrass model of $E/\K$ fails 
to stay minimal over $\F$, 
and $\Om_\phi(\F)$ is zero if the models
of $E$ and $E'$ change by the same amount. The relation to minimal 
discriminants is as follows:

\begin{notation}
Let $\phi: E\to E'$ be an isogeny of elliptic curves over $\K$. Write
$$
  \delta = v_\K(\Delta_{E/\K}), \quad
  \delta' = v_\K(\Delta_{E'\!/\K}), \quad
  \delta_\F = v_\F(\Delta_{E/\F}), \quad
  \delta'_\F = v_\F(\Delta_{E'\!/\F})
$$
for the valuations of the minimal discriminants.
\end{notation}

\begin{lemma}
\label{lemom}
If $\F/\K$ has ramification degree $e$, then
\beq
  \Om_\phi(\F)&=&
    \displaystyle
    \frac{e\delta'\!-\!\delta'_\F}{12}-\frac{e\delta\!-\!\delta_\F}{12}&=& 
    \displaystyle
    \frac {e(\delta'\!-\!\delta)}{12} - \frac{\delta'_\F\!-\!\delta_\F}{12}.    
\eeq
\end{lemma}

\begin{proof}
Using \cite{Sil1} Table III.1.2, we find that 
$$
  v_\F\Bigl(\frac{\omega_{E/\F}^{\min}}{\omega_{E/\K}^{\min}}\Bigr)=\frac{1}{12}\>
  v_\F\Bigl(\frac{\Delta_{E/\K}}{\Delta_{E/\F}}\Bigr)
    = \frac{e\delta\!-\!\delta_\F}{12},
$$
and similarly for $E'$. 
\end{proof}

\begin{theorem}
\label{ommain}
Let $\F/\K$ be a finite extension of $l$-adic fields of ramification degree $e$,
and $\phi: E\to E'$ an isogeny of elliptic curves over $\K$.
Then
%
$$
  \Om_\phi(\F) = e\mu + \epsilon(\F),
$$
and
\begin{itemize}
\item[(1)]
$\mu=\epsilon(\F)=0$ if $l\nmid\deg\phi$, or $E/\K$ has good, potentially ordinary 
or potentially multiplicative reduction.
\end{itemize}
Suppose that $E$ has additive potentially good reduction.
Write $\delta=v_\K(\Delta_{E/\K})$, $\delta'=v_\K(\Delta_{E'\!/\K})$.
Write $\eth=0,2,3,4,6,8,9,10$ if $E$ has Kodaira type 
\In{0}, \II, \III, \IV, \InS{n\ge 0}, \IVS, \IIIS, \IIS{} respectively, 
and similarly $\eth'$ for $E'$. Then
\begin{itemize}
\item[(2)]
$\mu=\frac{\delta'-\delta}{12}$, $\epsilon(\F)=\{\frac{e\delta}{12}\}-\{\frac{e\delta'}{12}\}$ 
if $E$ has tame reduction.
\item[(3)]
$\mu=\frac{\eth'-\eth}{12}$,
$\epsilon(\F)=\{\frac{e\eth}{12}\}-\{\frac{e\eth'}{12}\}$ 
if $\F/\K$ is tamely ramified.
\item[(4)]
$\mu=\frac{\delta'-\delta}{12},\>
|\epsilon(\F)|\le\frac23$ if $l\ne 2$.
If, moreover, $3|e$, then $|\epsilon(\F)|\le\frac12$.
\item[(5)]
$\mu=\frac{\delta'-\delta}{12}$, 
$|\epsilon(\F)|<\frac{re_{\L/\K}+1}2$
if $l=2$. Here $r$ is any real number satisfying $r\!>\!\frac{f_{E/\K}}2\!-\!1$,
where $f_{E/\K}$ is the conductor exponent of $E$, and 
$\L$ is the subfield of $\F$ cut out by the upper ramification group~$I_\K^r$.
\end{itemize}
\end{theorem}

\begin{proof}
By Lemma \ref{lemom},
$$
  \Om_\phi(\F)\>\>=\>\>
    \frac{e\delta'\!-\!\delta'_\F}{12}-\frac{e\delta\!-\!\delta_\F}{12}\>\>=\>\> 
    \frac {e(\delta'\!-\!\delta)}{12} + \frac{\delta_\F\!-\!\delta'_\F}{12}.    
    \eqno{(\dagger)}
$$
The claims are trivial if $\phi$ is an endomorphism $E\to E$. 
Decomposing the isogeny if necessary, 
it is clear that in (1)-(3) we may assume that $\deg\phi=p$ is prime.

(1) 
Theorem 5.1 of \cite{isogloc} (or \cite{isogloc} Table 1)
describes the change in the discriminant under isogenies of prime degree. 
If $E$ has potentially good reduction, and either $l\ne p$ or 
$E$ has good or potentially ordinary reduction, then
$\delta=\delta', \delta_\F=\delta'_\F$. 
If $E$ has potentially multiplicative reduction, then 
$\delta'-\delta=v_\K(j_E)-v_\K(j_{E'})$, 
and similarly
$\delta'_\F-\delta_\F=e(v_\K(j_E)-v_\K(j_{E'}))$. 
In both cases the right-hand side of $(\dagger)$ is 0.

(2) If $E/\K$ has tame reduction, the reduction stays tame over $\F$.
Furthermore, $0\le \delta,\delta',\delta_\F,\delta'_\F<12$ 
by \cite{isogloc} Thm. 3.1. Because the discriminant changes by 12th powers 
when changing the model, 
$$
  \delta_\F = 12\{ \frac{e\delta}{12} \}, \qquad
  \delta'_\F = 12\{ \frac{e\delta'}{12} \}, 
$$
and $(\dagger)$ implies the asserted formula.

(3) By \cite{tate} Thm. 3,
$$
  \delta_\F=e\delta - 12 \lfloor \frac{e \eth}{12} \rfloor, \qquad
  \delta'_\F=e\delta' - 12 \lfloor \frac{e \eth'}{12} \rfloor,
$$
and the claim follows from $(\dagger)$.

(4) Write $m, m'$ for the number of components of the N\'eron minimal models
of $E/\F$ and $E'\!/\F$. As $l\ne 2$, by \cite{Sil2} \S IV.9 Table 4.1
the curves $E, E'$ do not have Kodaira type $\InS{n>0}$, and $1\le m,m'\le 9$.
Since $E$ and $E'$ have the same conductor exponent (over $\F$), by Ogg's formula 
(\cite{Sil2} IV.11.1)
$$
  |\delta_\F-\delta'_\F| = |m-m'|\le 8.
$$
Therefore by $(\dagger)$,
$$
  \Om_\phi(\F)
    = \frac {e(\delta'-\delta)}{12} + \epsilon(\F), \qquad |\epsilon(\F)|\le \frac 23.
$$
If, moreover, $3|e$, then $3|\delta_\F, \delta'_\F$. In this case 
$|\delta_\F-\delta'_\F|\le 6$ and $|\epsilon(\F)|\le \frac 12$.

(5) 
By Ogg's formula and \cite{Pap} Thm. 2, we have
$f_{E/\F}\le\delta_\F\le 4f_{E/\F}$. As $E$ and $E'$ have the same conductor,
the same bounds hold for $\delta'_\F$, and so 
$|\delta_\F-\delta'_\F|\le 3f_{E/\F}$. By Theorem \ref{ellcond} in the appendix,
$$
  f_{E/\F} \>=\> f_{E/\L} \>\le\> e_{\L/\K} (f_{E/\K}-2)+2.
$$
Therefore
$$
  \frac{|\delta_\F-\delta'_\F|}{12}  \>\>\le\>\>
  \frac{e_{\L/\K}}4(f_{E/\K}-2)+\frac12 \>\><\>\>
  \frac{e_{\L/\K}}4\cdot 2r+\frac12 \>\>=\>\> \frac{re_{\L/\K}+1}2.
$$ 
The claim follows from $(\dagger)$, 
with $\epsilon(\F)=\frac{\delta_\F-\delta'_\F}{12}$.
\end{proof}

%


\begin{remark}
Note that in Theorem \ref{ommain} (3), the formula
$\mu=\frac{\eth'-\eth}{12}$ may \emph{not} be replaced by
$\frac{\delta'-\delta}{12}$. 
For example, the 2-isogenous curves 64A1, 64A4 over $\Q_2$
have type $\InS{2}$, $\delta=12$, $\eth=6$ and type $\II$, $\delta'=6$, $\eth'=2$, 
respectively, and $\delta'-\delta\ne \eth'-\eth$.
If $E/\K$ has tame reduction, e.g. if $\K$ has residue characteristic $\ge 5$, 
then $\delta=\eth$ and the two formulae are the same.
\end{remark}

\begin{corollary}
\label{alphacyc}
Let $\K_n=\Q_l(p^n)$ be the completion of the
$n$th layer of the~\hbox{$p$-cyclotomic} tower at a prime above~$l$,
and let $\phi: E\to E'$ be an isogeny of elliptic curves over $\Q_l$.
Then
$$
  \Om_\phi(\K_n)=p^n\mu + \epsilon(n),
$$
and $\mu=\epsilon(n)=0$ unless $l=p\,|\deg\phi$ and $E$ has additive potentially
super\-singular reduction. In this exceptional case,
%
%
writing $\delta=\ord_l\Delta_{E/\Q_l}$ and $\delta'=\ord_l\Delta_{E'\!/\Q_l}$,
we have
$$
  \mu=\frac{\delta'-\delta}{12}, \qquad 
  |\epsilon(n)|\le
  \left\{
  \begin{array}{llllll}
  \sfrac23 & \text{if $l\ge 5$}\cr
  \sfrac12 & \text{if $l=3$}\cr
  6\sfrac12 &     \text{if $l=2$.}\cr
  \end{array} 
  \right.
$$
If, moreover, $E$ has tame reduction, then 
$\epsilon(n)=\{\frac{p^n\delta}{12}\}-\{\frac{p^n\delta'}{12}\}$.
\end{corollary}

\begin{proof}
If $l\ne p$, then $\Om_\phi=0$ by Lemma \ref{semom0}(2).
If $l\nmid\deg\phi$, or $E/\K$ has good, potentially ordinary 
or potentially multiplicative reduction, then $\Om_\phi=0$ as well, by 
Theorem \ref{ommain}(1). Assume henceforth that $l=p\,|\deg\phi$ and 
$E$ has additive potentially supersingular reduction.

The last claim for $\epsilon(n)$ is contained in Theorem \ref{ommain} (2),
and the assertion for $l\ge 3$ is proved in Theorem \ref{ommain} (4).

Finally, suppose $l=2$.
By Theorem \ref{ommain} (5), the asserted equality holds 
with $|\epsilon(n)|< \frac{re+1}2$, where $r>\frac{f_{E/\Q_2}}2-1$ and
$e$ is the ramification degree of the field cut out by $I_{\Q_2}^r$.
By \cite{LRS}, we have $f_{E/\Q_2}\le 2+6v_{\Q_2}(2)=8$. 
So we can take $3<r<4$, $\L=\Q_2(2^2)$ and $e=4$
(or $\L=\Q_2(2)$, $e=2$ if $n=1$). Then
$|\epsilon(n)|<\frac{re+1}2\le \frac{4r+1}2$, 
which gives the bound of 6.5 as $r\to 3$.
\end{proof}

\begin{remark}
\label{om3}
If $l=p>3$, then every curve over $\K$ with potentially good reduction
is tame, so $\epsilon(n)$ in Corollary \ref{alphacyc} is explicit.
If $l\ne p$, it is zero, and for $l=p=3$ it can be made explicit as well
using Ogg's formula as follows:

Suppose $E/\K_n$ has wild reduction for all $n$.
The conductor exponent of $E/\K_n$ stabilises by 
Corollary \ref{corliestab} in the appendix, and 
the number of components $m$ on the N\'eron minimal model satisfies
$1\le m\le 9$ (cf. proof of Thm \ref{ommain} (4)). 
Therefore the congruence $\delta_{\K_n}\equiv e_{\K_n/\Q_l} \delta\mod 12$ 
determines both $m$ and $\delta_{\K_n}$ uniquely. The same is true for $E'$, 
and we get an explicit formula for 
$
  \epsilon(n)=\frac{\delta_{\K_n}-\delta'_{\K_n}}{12}
$
for large $n$.

Let us also observe that $\epsilon(n)\in\{-\frac 12,0,\frac 12\}$ for large $n$
in this case.
Indeed, the ramification degree $e_{\K_n/\Q_l}$ is divisible by 3,
and so $\delta_{\K_n}-\delta'_{\K_n}$ is also divisible by 3. 
By Lemma~\ref{lemper} and Theorem \ref{ommain} (4), 
$$
  |\frac{\delta_{\K_n}-\delta'_{\K_n}}{12}|=|\epsilon(n)|\le\frac 12,
$$
and so $|\delta_{\K_n}\!-\!\delta'_{\K_n}|\le 6$.
Moreover, $\delta_{\K_n}\!-\!3\delta'_{\K_n}\equiv 0\mod 4$ 
(\cite{isogloc} Thm~1.1), so 
they have the same parity, whence
$\delta_{\K_n}-\delta'_{\K_n}\in\{-6,0,6\}$.

%
\end{remark}

\section{Tamagawa numbers}
\label{stam}

We now turn to the behaviour of the Tamagawa number quotient 
from Corollary \ref{selmain} in towers of $l$-adic fields.

\begin{theorem}
\label{tammain}
Let $\K=\K_0\subset \K_1\subset \K_2\subset\ldots$ be a tower of $l$-adic fields, 
and $\phi: E\to E'$ an isogeny of elliptic curves over~$\K$.
Then the sequence
$$
  \frac{c_{E'\!/\K_0}}{c_{E/\K_0}}, \quad
  \frac{c_{E'\!/\K_1}}{c_{E/\K_1}}, \quad
  \frac{c_{E'\!/\K_2}}{c_{E/\K_2}}, \ldots
$$
stabilises unless $l|\deg\phi$, $e_{\K_n/\K}\to\infty$ 
and $E/\K$ has wild potentially supersingular reduction (in particular $l=2,3$). 
In this exceptional case, and in all other cases when $E$ has potentially good
reduction, all terms in the sequence are in $\{1,2,3,4,\frac12,\frac13,\frac14\}$.
\end{theorem}

\begin{proof}
First, if $e_{\K_n/\K}\not\to\infty$, then the extensions $\K_{n+1}/\K_n$ are 
eventually unramified, so the Tamagawa numbers of $E$ and $E'$ stabilise.

Next, suppose either $l\nmid \deg\phi$ or $E$ is not wild 
potentially supersingular. If $E=E'$ there is nothing to prove. Otherwise,
by decomposing $\phi$ into endomorphisms and isogenies of prime degree 
if necessary, we may assume that $\deg\phi=p$ is prime. Now apply the 
classification for the quotient $\frac{c}{c'}$ from \cite{isogloc} Table 1.
All the conditions in the table that determine $\frac{c}{c'}$ 
(e.g. that $E$ is good, ordinary, additive,
split multiplicative, has non-trivial 3-torsion, $v(j_E)=p v(j_{E'})$, etc.) 
stabilise in the tower $\K_n$, and hence so does the quotient. 

Now suppose that $E, E'$ have potentially good reduction. In particular, the
reduction is good or additive over all $\K_n$, and so 
$1\le c_{E/\K_n}, c_{E'\!/\K_n}\le 4$
(\cite{Sil2} \S IV.9, Table~4.1).
To prove the claim, it suffices to check
that if one of them, say $c_{E/\K_n}$ is 3, the other cannot be 2 or 4.
Indeed, $c_{E/\K_n}=3$ implies that $E$ has Kodaira type \IV{} or \IVS{}
and $c_{E'/\K_n}\in\{2,4\}$ that $E'$ has 
type $\III,\IIIS,\IZS,\InS{n}$ (loc. cit.).
Such curves cannot be isogenous, since (a) for $l=2$ the former types are tame
and the latter are wild (\cite{KT} Prop 8.20), 
(b) for $l=3$ the former are wild
and the latter are tame (\cite{Kra} Thm.~1), 
(c) for $l>3$ the curve $E'$ also has type \IV{} or \IVS{},
e.g. from \cite{isogloc} Table 1 or by considering the valuations of 
minimal discriminants and the smallest fields where the curves 
acquire good reduction.
%
%
\end{proof}

\begin{remark}
If $\deg\phi$ is a prime $p\ge 5$, then the quotient $\frac{c(E'\!/\K_n)}{c(E/\K_n)}$ 
is particularly simple: it stabilises to

\begin{itemize}
\itemindent -1em
\item 
$p$ if $E/\K_n$ has split multiplicative reduction for some $n$, and 
    $\frac{v_\K(j_{E'})}{v_\K(j_{E})}=p$.
\item 
$\frac 1p$ if $E/\K_n$ has split multiplicative reduction for some $n$, and 
    $\frac{v_\K(j_{E'})}{v_\K(j_{E})}\!=\!\frac 1p$.
\item 1 in all other cases.    
\end{itemize}
\end{remark}

\begin{corollary}
\label{ellcyctam}
If $\phi: E\to E'$ is an isogeny of elliptic curves over $\Q$, 
and $\Q(l^n)$ is the $n$th layer in the $l$-cyclotomic tower, then the sequence
$$
  \prod_v\frac{c_v(E'\!/\Q(l^n))}{c_v(E/\Q(l^n))} 
$$
stabilises, unless $l|\deg\phi$ and $E$ has wild potentially supersingular
reduction at $l$ (in particular, $l=2,3$). In this exceptional case,
for all sufficiently large $n$,
the terms are of the form $C\cdot \alpha_n$ for some constant $C$ and some
$\alpha_n\in \{1,2,3,4,\frac12,\frac13,\frac14\}$.
\end{corollary}

\begin{remark}
For elliptic curves with wild potentially supersingular reduction,
the quotient of Tamagawa numbers might {\em not} stabilise. 

As an example, take the 3-isogeny $E=243A1\to 243A2=E'$ over $\Q(3^n)$,
and consider the Tamagawa numbers at primes above 3.
As in Remark \ref{om3}, one may compute the minimal discriminants and
the number of components of the N\'eron minimal model (and thus the
Kodaira types). We find that the Kodaira types of $E, E'$ alternate between 
$\IVS, \II$ and $\IIS, \IV$. The Tama\-gawa numbers for types \II{} and \IIS{} are
always 1 (\cite{Sil2} \S IV.9, Table 4.1), and they turn out to be 3 
for the $\IV, \IVS$ cases, so that the Tamagawa quotient alternates between
$3$ and $\frac 13$.
(To see that they are 3, we use the fact that 
over a local field $\K/\Q_3$,
the parity of 
$\ord_3\frac{c(E'/\K)}{c(E/\K)}$ can be recovered $\Om_\phi(\K)$, the local root 
number $w(E/\K)$, and the Artin symbol $(-1,\K(\ker\phi^t)/\K)$; see 
\cite{kurast}~Thm 5.7. In our case, $\Om_\phi(\K)$ is computed as in 
Remark \ref{om3}, the local root number is $+1$ over $\Q_3$ and is unchanged 
in odd degree Galois extensions, and the points in $\ker\phi^t$ 
are defined over $\Q$, so that the Artin symbol is trivial.)
%
\end{remark}

Finally, we record the fact that Theorem \ref{tammain} also holds for 
semistable abelian varieties:

\begin{theorem}
Let $\K\subset \K_1\subset \K_2\subset\ldots$ be a tower of $l$-adic fields, 
and $\phi: A\to A'$ an isogeny of semistable abelian varieties over $\K$.
Then the quotient of Tamagawa numbers $\frac{c(A'\!/\K_n)}{c(A/\K_n)}$ 
stabilises as $n\to\infty$.
\end{theorem}

\begin{proof}
By a result of \cite{BD}, for sufficiently large $n$
$$
  c(A/\K_n) = C\cdot e_{\K_n/\K}^r 
     \qquad\text{and}\qquad
  c(A'\!/\K_n) = C'\cdot e_{\K_n/\K}^r,
$$
for some constants $C, C'$. Here $e$ is the ramification degree, and $r$ denotes 
the rank of the split toric part of $A/\K_n$ (and of $A'/\K_n$) for large enough $n$. 
The claim follows. 
\end{proof}

\section{Torsion}
\label{stors}

To deduce the exact growth of
$\sha$ from the isogeny invariance of the Birch--Swinnerton-Dyer conjecture,
we need to control torsion in the Mordell--Weil group in towers of number fields.
This is the purpose of this section.

\begin{proposition}
\label{torsstab}
Let $K\subset K_1\subset K_2\subset\ldots$ be a tower
of number fields, and $\phi: A\to A'$ an isogeny 
of abelian varieties over~$K$ of degree $p^km$, $p\nmid m$. 
Then for every $n\ge 1$,
$$
  \frac{|A(K_n)[p^\infty]|}{|A'(K_n)[p^\infty]|} = p^{a_n}\qquad\text{with}\qquad
    k(1-2\dim A)\le a_n\le k.
$$
If $A$ has finite $p$-power torsion over $\bigcup_n K_n$, then the quotient
stabilises.
\end{proposition}

\begin{proof}
Because $\phi: A(K_n)[p^\infty] \to A'(K_n)[p^\infty]$ has kernel of size at 
most $p^k$, the left-hand quotient is at most $p^k$. The same argument applied to
the conjugate isogeny $\phi': A'\to A$ (with $\phi\circ\phi'$ 
the multiplication-by-$p^k$ map) gives the first claim.
The second claim is clear, as both $A(K_n)[p^\infty]$ and $A'(K_n)[p^\infty]$
stabilise.
\end{proof}

\begin{corollary}
\label{cyctorsstab}
Let $K$ be a number field, $K_\infty=\bigcup_n K_n$ its cyclotomic\linebreak
$\Z_l$-extension,
and $A,A'\!/K$ isogenous abelian varieties. Then for every prime $p$ 
the sequence
$\frac{|A(K_n)[p^\infty]|}{|A'(K_n)[p^\infty]|}$ stabilises as $n\to\infty$.
\end{corollary}

\begin{proof}
By the Imai-Serre theorem (see \cite{Rib}), 
$A$ has finite torsion over $K_\infty$. 
\end{proof}

\pagebreak

\begin{remark}
\label{elltorscyc}
If $A\!=E$ is an elliptic curve over $K\!=\!\Q$, and $K_n/\Q$~are \hbox{Galois},~the 
assumption 
that $E$ has finite $p$-power torsion over $K_\infty=\bigcup_n K_n$ simply means
that $K_\infty$ does not contain the full $p$-division tower $\Q(E[p^\infty])$.
Indeed, if $E/K_\infty$ has infinite $p$-power torsion and $K_\infty/\Q$ is 
Galois but does not contain $\Q(E[p^\infty])$, 
then $E[p^n]$ has a Galois stable cyclic subgroup of order $p^n$
and so a cyclic $p^n$-isogeny for every $n\ge 1$. 
Since $E$ cannot have CM over $\Q$, this is impossible 
by Shafarevich's theorem on the finiteness of isogeny classes.
\end{remark}

\section{Divisible Selmer}
\label{sdiv}

The ultimate global invariant that we need to control is the divisible 
part\linebreak of the $p^\infty$-Selmer group. 
Its $\Z_p$-corank $\rk_p A/K$ is conjecturally the Mordell--Weil rank,
which is hard to bound in general towers of number fields $K_n/K$.
As a result, we only give elementary bounds on 
$|\Sel^{\div}_{p^\infty}(A/K_n)[\phi]|$ in terms of $\rk_p A/K_n$
and prove some stabilisation results that we will need in \S\ref{sgrowth}.
%


\begin{lemma}
\label{rkplem}
Let $\phi: A\to A'$ be an isogeny of abelian varieties over a number field $K$,
and $p$ a prime number. Then 
$$
  |\Sel^{\div}_{p^\infty}(A/K)[\phi]| \le p^{\rk_p A/K\>\cdot\>\ord_p\deg\phi}.
$$
\end{lemma}

\begin{proof}
Let $d=\deg\phi$, and let $\phi'$ be the conjugate isogeny, so that 
$\phi'\circ\phi$ is multiplication by $d$. Then it is clear that
$$
  \Sel^{\div}_{p^\infty}(A/K)[\phi] \subset \Sel^{\div}_{p^\infty}(A/K)[d].
$$ 
The right-hand side has size $p^{\rk_p A/K\cdot\>\ord_p d}$, since
$\Sel^{\div}_{p^\infty}(A/K)$ is isomorphic to $(\Q_p/\Z_p)^{\rk_p A/K}$.
\end{proof}

\begin{corollary}
\label{rkpcor}
Let $K\!\subset\!K_1\!\subset\!K_2\!\subset\!\ldots$ be a tower
of number fields, $\phi:\! A\!\to\!A'$ an isogeny of abelian varieties over $K$,
and $p$ a prime number. Then
$$
\ord_p\frac{
|\Sel^{\div}_{p^\infty}(A/K)[\phi]|
}{
|\Sel^{\div}_{p^\infty}({A'}^t/K)[\phi^t]|
} 
= O (\rk_p A/K_n).
$$
\end{corollary}

\begin{proof}
Apply Lemma \ref{rkplem} to $\phi$ and $\phi^t$. 
\end{proof}

\begin{lemma}
\label{vovscontrol}
Let $F/K$ be a Galois extension of number fields, and $A/K$ an abelian variety.
Then 
$$
  |\ker(\Res:\Sel_{p^\infty}^{\div}(A/K)\to \Sel_{p^\infty}^{\div}(A/F))| 
     \le |A(F)[p^\infty]|^{\rk_p A/K}.
$$
\end{lemma}

\begin{proof}
First observe that for every $n\ge 1$, 
$$
  \ker(\Res: \Sel_{p^n}(A/K)\to \Sel_{p^n}(A/F))
$$
is killed by $M=|A(F)[p^\infty]|$. 
Indeed, $\Sel_{p^n}(A/K)\subset H^1(K,A[p^n])$, and 
the kernel of $\Res: H^1(K,A[p^n])\to H^1(F,A[p^n])$ is
$H^1(\Gal(F/K),A(F)[p^n])$ by the inflation-restriction sequence,
and so is clearly killed by $M$.
Because $\Sel_{p^\infty}$ is the injective limit of the $\Sel_{p^n}$, 
the asserted kernel on $\Sel_{p^\infty}$ and on $\Sel_{p^\infty}^{\div}$ 
is also killed by $M$.
As $\Sel_{p^\infty}^{\div}(A/K)\iso (\Qp/\Zp)^{\rk_p A/K}$, 
the result follows.
%
%
\end{proof}

\begin{theorem}
\label{divselstab}
Let $K$ be a number field, $A/K$ an abelian variety, $p$ a prime number and 
$K\subset K_1\subset K_2\subset\ldots$ a tower of Galois extensions of~$K$. 
Suppose that both $A(K_n)[p^\infty]$ and 
$\rk_p A/K_n$ are bounded as $n\to\infty$.
\begin{enumerate}
\item 
There exists $n_0$ such that the restriction maps
$$
  J_{n,n'}: \Sel^{\div}_{p^\infty}(A/K_n) \lar \Sel^{\div}_{p^\infty}(A/K_{n'})
$$
are isomorphisms for all $n_0\le n\le n'$. 
\item
If $\phi: A\to A'$ is an isogeny, 
then $|\Sel^{\div}_{p^\infty}(A/K_n)[\phi]|$
is eventually constant, equal to $p^\lambda$ for some 
$0\le \lambda\le \ord_p\deg\phi\cdot\lim_{n\to\infty}\rk_p A/K_n$.
\end{enumerate}
\end{theorem}

\begin{proof}
(1) Because $\rk_p A/K_n\le\rk_p A/K_{n'}$ for $n\le n'$, the 
$p^\infty$-Selmer rank eventually 
stabilises. Replacing $K$ by some $K_m$, we may thus assume 
that $\rk_p A/K_n$ is independent of $n$. 
The maps $J_{n,n'}$ have finite kernels by Lemma~\ref{vovscontrol}, so they 
must then be surjective.

By assumption, $A$ has finite $p$-power torsion over $\bigcup K_n$,
say of order~$M$. 
The (increasing) sequence $|\ker J_{0,n}|$ is bounded
by Lemma \ref{vovscontrol}, so it stabilises, say at $n_0$. 
Now suppose $n_0\le n\le n'$. 
The map $J_{n,n'}$ cannot have non-trivial kernel because 
$J_{0,n'}=J_{n,n'}\circ J_{0,n}$, the kernels 
of $J_{0,n}$ and of $J_{0,n'}$ are of the same size, and 
$J_{0,n}$ is surjective.
Thus the $J_{n,n'}$ are both surjective and injective, as required.

(2) Take $n_0$ as above, so that $\rk_p A/K_n$ is constant and $J_{n_0,n}$ 
are isomorphisms for $n\ge n_0$.
Then, by Lemma \ref{rkplem},
$|\Sel^{\div}_{p^\infty}(A/K_{n_0})[\phi]|=p^\lambda$ for some 
$0\le\lambda\le \ord_p\deg\phi\cdot\rk_p A/K_{n_0}$.
Now the maps $J_{n_0,n}$ are isomorphisms and commute with $\phi$,
and the result follows. 
\end{proof}

\begin{corollary}
Let $K_\infty\!=\!\cup_n K_n$ be the $l$-cyclotomic tower of a 
number field~$K$, and let $\phi: A\to A'$ be an isogeny of abelian varieties 
over $K$. If $\rk_p A/K_n$ is bounded as $n\to\infty$, then
$|\Sel^{\div}_{p^\infty}(A/K_n)[\phi]|$ is eventually constant.
\end{corollary}

\begin{proof}
For cyclotomic towers, torsion in $A(K_n)$ is bounded by the 
Imai-Serre theorem (see \cite{Rib}).
\end{proof}

\begin{corollary}
\label{corselellcyc}
Let $\Q_\infty\!=\!\cup_n \Q(l^n)$ be the $l$-cyclotomic tower of $\Q$, 
and let $\phi:\!E\!\to\!E'$ be an isogeny of elliptic curves over $\Q$.
Then $|\Sel^{\div}_{p^\infty}(E/\Q(l^n))[\phi]|$ is eventually constant.
\end{corollary}

\begin{proof}
By Kato's theorem \cite{KatP},
the $p^\infty$-Selmer rank of $E$ is bounded in the cyclotomic tower.
\end{proof}

\section{Selmer growth in towers}
\label{sgrowth}

In this section we prove Theorems \ref{imain}-\ref{iLie}. 
This relies on the invariance of 
the Birch--Swinnerton-Dyer conjecture under isogeny (as in \S\ref{sinv})
and the explicit computations of periods $\Omega_A, \Omega_{A'}$, 
exterior form contributions $\Om_\phi$,
Tamagawa numbers and torsion 
from \S\ref{sper}, \S\ref{som}, \S\ref{stam} and \S\ref{stors}.
In fact, Theorem~\ref{main} and Proposition \ref{gammamain} give an explicit
description of the Selmer quotient in almost completely 
general towers of number fields.
To obtain the statements for $\sha$, we also bound the contributions from
the divisible part of Selmer using the results of \S\ref{sdiv}.

\begin{notation}
\label{notgammav}
Let $F/K$ be a Galois extension of number fields, 
and let $\phi: A\to A'$ be an isogeny of abelian varieties over $K$.
For a place $v$ of $K$ and a place $w|v$ of $F$ write 
$$
\gamma_v=\left\{
\begin{array}{llll}
 \frac{\Omega^\ast_{A'\!/K_v}}{\Omega^\ast_{A/K_v}} & 
    \text{if }K_v\iso\R, F_w\iso\C\cr
 1 & \text{in all other Archimedean cases} \cr
 \frac{c_w(A'\!/F)}{c_w(A/F)}|k_w|^{\Om_\phi(F_w)} & \text{if } v\nmid\infty, \cr
\end{array}
\right.
$$
where $k_w$ is the residue field at $w$ 
and $\Om_\phi$ is as in Definition \ref{defom}. 
Note that $\gamma_v=1$ for primes of $v$ of good reduction, 
see Lemma \ref{semom0}.

If $K\subset K_1\subset K_2\subset ...$ is a tower of Galois extensions of $K$,
we write $e_{v,n}, f_{v,n}, n_{v,n}$ for the ramification degree of $v$,
the residue degree of $v$ and the number of places above $v$ in $K_n/K$.
In this setting, we also write $\gamma_{v,n}$ for the $\gamma_v$ for $F=K_n$.
\end{notation}

\begin{theorem}
\label{main}
Let $F/K$ be a Galois extension of number fields, 
and let $\phi: A\to A'$ be an isogeny of abelian varieties over $K$.
If the degree of $\phi$ is a power of $p$, then
\begingroup\smaller[1]
$$
\frac{
|\Sel^{\div}_{p^\infty}(A/F)[\phi]|
}{
|\Sel^{\div}_{p^\infty}(A'^t/F)[\phi^t]|
} 
  \frac{|\shazero_{A/F}[p^\infty]|}{|\shazero_{A'\!/F}[p^\infty]|} =
  \frac{|A(F)[p^\infty]||A^t(F)[p^\infty]|}{|A'(F)[p^\infty]||A'^t(F)[p^\infty]|}
%
\Bigl(\frac{\Omega_{A'\!/K}}{\Omega_{A/K}}\Bigr)^{[F:K]}
\prod_v \gamma_v^{n_v},
$$
\endgroup
where $n_v$ is the number of places of $F$ above $v$.
If $\phi$ has arbitrary degree, then
the left-hand side and the right-hand side have the same $p$-part.
\end{theorem}

\begin{proof}
This a rephrasification of Corollary \ref{selmain}.
\end{proof}

\begin{theorem}
\label{cycmain}
Let $\Q_\infty\!=\!\cup_n \Q(l^n)$ be the $l$-cyclotomic tower,
$p$ a prime, and $\phi: E \to E'$ an isogeny of elliptic curves over $\Q$. 
Then for all large enough~$n$,
$$
\frac{
|\Sel^{\div}_{p^\infty}(E/\Q(l^n))[\phi]|
}{
|\Sel^{\div}_{p^\infty}(E'\!/\Q(l^n))[\phi^t]|
} 
  \frac{|\shazero_{E/\Q(l^n)}[p^\infty]|}{|\shazero_{E'\!/\Q(l^n)}[p^\infty]|} =
%
  p^{\,\mu l^n + \kappa + \epsilon(n)},
$$
with $\mu\in\frac{1}{12}\Z$ given by
$$
  \mu = \ord_p\frac{\Omega_{E'}}{\Omega_E} + 
  \left\{\begin{array}{llllll}
     0&{\text{if $\,l\ne p\>$ or $\>\ord_p(j_E)<0$,}}\\[3pt]
     {\frac1{12}{\ord_p(\frac{\Delta_{E'}}{\Delta_E})}}&{\text{if $\,l=p\>$ and $\>\ord_p(j_E)\ge 0$,}}\cr
  \end{array}
  \right.
$$
some $\kappa\in\Z$, and 
$|\epsilon(n)|\!\le\sfrac23$ for $p\!>\!3$,
$|\epsilon(n)|\!\le\sfrac32$ for $p\!=\!3$ and 
$|\epsilon(n)|\!\le8\sfrac12$ for $p\!=\!2$.
If $l\ne p$ or $l\nmid\deg\phi$, then $\epsilon(n)=0$.
\end{theorem}

\begin{proof}
To find the Selmer quotient, we apply Theorem \ref{main}.
The torsion contribution is eventually constant by Corollary \ref{cyctorsstab}.
Now compute the $\gamma_{v,n}$ (see Notation \ref{notgammav}) 
for $K_n=\Q(l^n)$ and all primes $v$ of $\Q$.
By Corollary \ref{ellcyctam}, the product of Tamagawa quotients stabilises
unless $l\le 3$, $l|\deg\phi$ and $E$ 
\linebreak\pagebreak\par\noindent
has wild potentially 
supersingular reduction at $l$.
In that case, the term is 
of the form $l^{\xi(n)}$ with 
$|\xi(n)|\le 1$ when $l=3$ and $|\xi(n)|\le 2$ when $l=2$. 
By Corollary \ref{alphacyc}, the $\Om_\phi$-term is 1 
unless $l=p$ and $E$ has additive potentially
supersingular reduction at $l$, in which case it is 
$$
  p^{{\frac{l^n}{12}{\ord_p(\frac{\Delta_{E'}}{\Delta_E})}} + \eta(n)},
$$
with $|\eta(n)|$ bounded by $\sfrac23,\sfrac12,6\sfrac12$ for $p>3$,
$p=3$ and $p=2$, respectively. The claim follows, with the asserted bounds 
for $\epsilon(n)=\xi(n)+\eta(n)$.
\end{proof}

\begin{corollary}
\label{pf1.1}
Theorem \ref{imain} holds.
\end{corollary}

\begin{proof}
Combine Theorem \ref{cycmain} and Corollary \ref{corselellcyc}.
\end{proof}

\begin{proposition}
\label{gammamain}
Let $K\subset K_1\subset K_2\subset\ldots$ be a tower
of Galois extensions of~$K$, and $\phi: A\to A'$ an isogeny
of abelian varieties over~$K$. For a prime $v$ of $K$ and a rational prime $p$,
we have
\begin{itemize}
\item[(1)]
If $v\nmid p$, or $A$ is semistable at $v$, or $A$ is an elliptic curve which
does not have additive potentially supersingular reduction at $v$,
then $\ord_p\gamma_{v,n}$ is constant for sufficiently large $n$.
\end{itemize}
Suppose that $v|p$ and $A=E, A'=E'$ are elliptic curves with potentially good 
reduction at $v$. Write 
$\delta=\delta_{E/K_v}, \delta'=\delta_{E'\!/K_v}, \eth=\eth_{E/K_v}, \eth'=\eth_{E'\!/K_v}$
as in Theorem \ref{ommain}.
Then, for all sufficiently large $n$,
$$
  \ord_p\gamma_{v,n} = \mu_v e_{v,n} f_{v,n} + 
    \epsilon_{v,n} f_{v,n} f_{K_v/\Q_p} + z_n, 
$$
with
\begin{itemize}
\item[(2)]
$\mu_v=f_{K_v/\Q_p}\frac{\delta'-\delta}{12}$, 
$\epsilon_{v,n}=\{\frac{e_{v,n}\delta}{12}\}-\{\frac{e_{v,n}\delta'}{12}\}$,
and $z_n$ is constant
if $E$ has tame reduction at~$v$.
\item[(3)]
$\mu_v=f_{K_v/\Q_p}\frac{\eth'-\eth}{12}$, 
$\epsilon_{v,n}=\{\frac{e_{v,n}\eth}{12}\}-\{\frac{e_{v,n}\eth'}{12}\}$,
$|z_n|\le 2$
if all $K_n/K$ are tamely ramified at $v$.
\item[(4)]
$\mu_v=f_{K_v/\Q_p}\frac{\delta'-\delta}{12}$, 
$|\epsilon_{v,n}|\le\frac 23$, $|z_n|\le 1$
if $p=3$.
\item[(5)]
$\mu_v=f_{K_v/\Q_p}\frac{\delta'-\delta}{12}$,
$\epsilon_{v,n}=O(1)$,  $|z_n|\le 2$
if $p=2$ and
all upper ramification subgroups of the inertia group 
$I_{\bigcup\!K_n/K}$ at $v$ have finite index.
\end{itemize}
\end{proposition}

\begin{proof}
Combine 
Lemma \ref{semom0}(1), 
Theorem \ref{ommain} and Theorem \ref{tammain}.
\end{proof}

\begin{remark}
\label{gammaconst}
For sufficiently large $n$,
the constants $\ord_p\gamma_{v,n}$ in (1) and $z_n$ in (2) 
are just $\ord_p\frac{c_w(A'\!/K_{n})}{c_w(A/K_{n})}$,
where $w$ is a prime of $K_n$ above $v$.
If $K_\infty/K_n$ is unramified at primes above $v$ for $n$ sufficiently large,
then $z_n=\ord_p\frac{c_w(A'\!/K_{n})}{c_w(A/K_{n})}$
and $\epsilon_{v,n}$ are constants in (3)--(5) as well,
by Lemma \ref{semom0}(2) and Theorem \ref{tammain}.
\end{remark}

\pagebreak

\begin{theorem}
\label{Zlmain}
Let $K$ be a number field, $p$ a prime, 
and $K_\infty\!=\!\cup_n K_n$ a \hbox{$\Zl$-extension} of $K$,
with $[K_n:K]=l^n$. Let $\phi: E\to E'$ be an isogeny of 
elliptic curves over~$K$. Then
$$
\frac{|\Sel^{\div}_{p^\infty}(E/K_n)[\phi]|}
     {|\Sel^{\div}_{p^\infty}(E'\!/K_n)[\phi^t]|} 
     \frac{|\shazero_{E/K_n}[p^\infty]|}{|\shazero_{E'\!/K_n}[p^\infty]|} 
         =
  p^{\,\mu l^n+O(1)},
  \quad 
  \mu=\ord_p \bigl(\frac{\Omega_{E'\!/K}}{\Omega_{E/K}}\bigr)+\sum_{v} \mu_v,
$$
where the sum is taken over the primes $v$ of bad reduction for $E/K$, and
\begingroup
\smaller[1]
$$
  \quad\mu_v = \left\{
    \begin{array}{llll}
    \ord_p\frac{c_v(E'\!/K)}{c_v(E/K)} & \text{if } v\text{ is totally split in $K_\infty/K$},& \cr
    \frac{f_{K_v/\Q_p}}{12}{\ord_v(\frac{\Delta_{E'\!/K}}{\Delta_{E/K}})}& \text{if $l\!=\!p,$\ $v|p$ is ramified in $K_\infty/K$
      and $\ord_v j_E\!\ge\!0$},\cr
    0 & \text{otherwise.} \cr
    \end{array}
  \right.
$$
\endgroup
\end{theorem}

\begin{proof}
%
We apply Theorem \ref{main}. 
The torsion contribution is $O(1)$ by Proposition \ref{torsstab}.
All Archimedean places are totally split in $K_\infty/K$,
since $\Z_l$ has no elements of order 2, so $\gamma_{v,n}=1$ for $v|\infty$.
It remains to show
that $\ord_p(\gamma_{v,n}^{n_{v,n}})=\mu_v l^n+O(1)$ for primes $v$ of $K$ 
of bad reduction for $E, E'$. 
Because the inertia and the decomposition subgroups at $v$ are closed 
subgroups of $\Z_l$, the possible local behaviours of $K_\infty/K$
at $v$ are
\begin{enumerate}
\item 
$K_\infty/K$ is totally split at $v$, so $n_{v,n}=l^n$, $e_{v,n}=1$, $f_{v,n}=1$.
\item
$K_\infty/K$ is not totally split but is unramified at $v$, 
so $n_{v,n}\!=\!c$, $e_{v,n}\!=\!1$, $f_{v,n}\!=\!l^n/c$ 
for some constant $c$ and all large enough $n$.
\item
$K_\infty/K$ is ramified at $v$,
so $n_{v,n}=c_1$, $e_{v,n}=l^n/c_1c_2$, $f_{v,n}=c_2$ 
for some constants $c_1, c_2$ and all large enough $n$; here necessarily $v|l$.
\end{enumerate}

In the first case, $\Om_\phi(K_{n,w})=0$ for $w|v$ since $K_{n,w}=K_v$, 
and hence $\ord_p(\gamma_{v,n}^{n_{v,n}})=\mu_v l^n$.
In the second case, $\Om_\phi$ is still 0 by Lemma \ref{semom0}(2), so
$\gamma_{v,n}^{n_{v,n}}$ is a finite product of bounded quotients of 
Tamagawa numbers, which is $O(1)$.

Assume we are in the third case.

If $v\nmid p$ or $E$ does not have additive potentially good reduction at $v$, 
then $\ord_p\gamma_{v,n}$ is eventually constant 
by Proposition \ref{gammamain} (1),
and hence so is $\ord_p(\gamma_{v,n}^{n_{v,n}})$; 
in particular, it is again $O(1)$.
Also $\mu_v=0$: either $v\nmid p$, or 
$E, E'$ have good reduction at $v$ and $\ord_v\Delta_{E/K}=\ord_v\Delta_{E'\!/K}=0$,
or $E, E'$ have potentially multiplicative reduction (i.e. 
$\ord_v j_E<0$).

Suppose $v|p$ and $E$ has additive potentially good reduction at $v$.
Now apply Proposition \ref{gammamain} (2) for $p>3$, (4) for $p=3$
and (5) for $p=2$ (noting that by class field theory,
all upper ramification groups in the inertia group of $\Gal(K_\infty/K)=\Z_l$ 
are of finite index) to find that
$$
  \ord_p(\gamma_{v,n}^{n_{v,n}})  =
   \bigl(f_{K_v/\Q_p}\tfrac{\delta'-\delta}{12} e_{v,n} f_{v,n} \!+\! 
    \epsilon_{v,n} f_{v,n} f_{K_v/\Q_p} \!+\! z_n\bigr)n_{v,n}
      = \mu_v l^n + O(1).
$$
\end{proof}

\begin{corollary}
\label{pf1.2}
Theorem \ref{iZlmain} holds.
\end{corollary}

\begin{proof}
Combine Theorem \ref{Zlmain} and Corollary \ref{rkpcor}.
\end{proof}
%


\pagebreak

\begin{theorem}
\label{Lie}
Suppose $K$ is a number field and $K_\infty/K$ a Galois extension 
whose Galois group is a $d$-dimensional $l$-adic Lie group; 
write $K_n/K$ for its $n$th layer in the natural Lie filtration. 
Let $p$ be a prime number, and $\phi: A\to A'$ a $K$-isogeny 
of abelian varieties. 
If either $A, A'$ are semistable or if they are
elliptic curves that do not have additive potentially supersingular 
reduction at primes $v|p$ that are infinitely ramified in $K_\infty/K$,
then there are constants $\mu_1,...,\mu_d\in\Q$ such that for all 
sufficiently large $n$,
$$
\frac{|\Sel^{\div}_{p^\infty}(A/K_n)[\phi]|}
     {|\Sel^{\div}_{p^\infty}(A'^t/K_n)[\phi^t]|} 
  \frac{|\shazero_{A/K_n}[p^\infty]|}{|\shazero_{A'\!/K_n}[p^\infty]|} 
  \frac{|A'(K_n)[p^\infty]||A'^t(K_n)[p^\infty]|}{|A(K_n)[p^\infty]||A^t(K_n)[p^\infty]|}
= \qquad \qquad 
$$
$$
\qquad \qquad \qquad \qquad \qquad \qquad \qquad \qquad \qquad 
=p^{\,\mu l^{dn} + \mu_1 l^{(d-1)n} + \ldots + \mu_{d-1}l^n + \mu_d}.
$$
If $A$ is a general elliptic curve, then there is
$\mu\in\Q$ such that for all sufficiently large $n$, 
$$
\frac{|\Sel^{\div}_{p^\infty}(A/K_n)[\phi]|}
     {|\Sel^{\div}_{p^\infty}({A'}/K_n)[\phi^t]|} 
  \frac{|\shazero_{A/K_n}[p^\infty]|}{|\shazero_{A'\!/K_n}[p^\infty]|} 
  = 
p^{\,\mu l^{dn} + O(l^{(d-1)n})}.
$$
\end{theorem}

\begin{proof}
We will apply Theorem \ref{main}. Note that the torsion contribution 
is $O(1)$ by Proposition \ref{torsstab}, so we can ignore it for the second
claim.

Fix a prime $v$ of $K$.
The decomposition group $D_v$ and the inertia group $I_v$ 
are closed Lie subgroups. So, for all sufficiently large~$n$,
$$
\begin{array}{llllllllllll}
  e_{v,n} &=& C_1 \> l^{n(\dim I_v)}, \cr
  f_{v,n} &=& C_2 \> l^{n(\dim D_v-\dim I_v)}, \cr
  n_{v,n} &=& C_3 \> l^{n(d-\dim D_v)}. \cr
\end{array}
$$
for some constants $C_1, C_2, C_3$ (that depend on $v$). 
If $v\nmid l$, then the tower $K_\infty/K_n$ is eventually tamely ramified at $v$,
and $\dim I_v\le 1$.
If $v|l$, then, on the contrary,
all the upper ramification subgroups of $I_v$ 
are of finite index by Sen's theorem (\cite{Sen} \S4, main Theorem).

We now compute $\ord_p(\gamma_{v,n}^{n_{v,n}})$ for every place $v$ of $K$:

\emph{$v$ Archimedean.} If $v$ is real and stays real in the tower, 
or $v$ is complex, then $\gamma_{v,n}=1$. Otherwise,  
$\gamma_{v,n}\in\Q$ stabilises for large enough $n$, 
so $\ord_p(\gamma_{v,n}^{n_{v,n}})$ grows like $C l^{dn}$ 
for a suitable $C\in\Q$.

\emph{$v\nmid p$ or $A/K_v$ is semistable or 
$A/K_v$ is an elliptic curve that
is not additive potentially supersingular.}
Again, $\ord_p\gamma_{v,n}$ is constant for large $n$, by 
Proposition \ref{gammamain} (1), so 
$\ord_p(\gamma_{v,n}^{n_{v,n}})$ grows like $C l^{n(d-\dim D_v)}$
for some $C\in\Q$.

\emph{$v|p$ and $A/K_v$ is an elliptic curve.}
By Proposition \ref{gammamain}, 
for large $n$,
$$
  \ord_p\gamma_{v,n} = \mu_v e_{v,n} f_{v,n} + 
    \epsilon_{v,n} f_{v,n} f_{K_v/\Q_p} + z_n, 
$$
with $z_n$, $\mu_v$ and $\epsilon_{v,n}$ as in the proposition. 
(The proposition applies because 
either $v|l$ and the upper ramification groups at $v$ have finite index in $I_v$,
or $v\nmid l$ and $K_\infty/K_n$ are eventually tamely ramified.)
If $e_{v,n}\to\infty$, then $\dim I_v\ge 1$ and 
$\ord_p(\gamma_{v,n}^{n_{v,n}}) = C l^{dn} + O(l^{(d-1)n})$ for some $C\in\Q$.
Otherwise, $e_{v,n}$ is eventually constant, as are 
$z_n$ and $\epsilon_{v,n}$ by Remark \ref{gammaconst}.
In that case $\ord_p(\gamma_{v,n}^{n_{v,n}}) = C l^{dn} + C' l^{n(d-\dim D_v)}$
for some $C,C'\in\Q$.

Taking the product over all places $v$, and applying 
Theorem \ref{main}, we get the claim. 
(The term $\bigl(\frac{\Omega_{A'\!/K}}{\Omega_{A/K}}\bigr)^{[K_n:K]}$
gives a contribution of the form $p^{C l^{dn}}$ for some $C\in\Q$.)
\end{proof}

\begin{corollary}
\label{pf1.3}
Theorem \ref{iLie} holds.
\end{corollary}

\begin{proof}
(1) This is Theorem \ref{Lie} (2).

(2) Combine Theorem \ref{Lie} (1) and Proposition \ref{torsstab}.

(3) 
If $\rk_p A/K_n\!=\!O(l^{(d-1)n})$, then the divisible Selmer quotient
is $p^{\scriptscriptstyle O(l^{(d-1)n})}$ by Corollary \ref{rkpcor}, and the torsion
quotient is $p^{O(1)}$ by Proposition \ref{torsstab}. 
Moreover, if $\rk_p A/K_n$ and $A(K_n)[p^\infty]$ are bounded, then the torsion 
quotients stabilize by Proposition \ref{torsstab}, 
as does the divisible Selmer quotient by \linebreak Theorem \ref{divselstab}(2)
applied to $\phi$ and to $\phi^t$. 
This gives the results for $\shazero$ .
\end{proof}

\section{Appendix. Conductors of elliptic curves in extensions}
\label{sapp}

In this appendix we give bounds on the growth of conductors of 
elliptic curves in extensions of local fields. These results are independent
of the rest of the paper, and do not rely on the presence of an isogeny.

Let $\K$ be a finite extension of $\Q_l$, and $I_\K$ the inertia
subgroup of $\Gal(\bar \K/\K)$. We will be interested in 
continuous representations $\chi: I_\K\to\GL_d(\C)$,
in other words those with finite image.
We write $I^n$ for the upper ramification groups of $I_\K$.
We denote by $f_\chi, f_E, ...$ the conductor exponent of a representation, 
an elliptic curve etc.

\begin{notation}
For an irreducible continuous representation $\chi\!\ne\!\triv$ of $I_\K$, set
$$
  m_\chi=\max_{i\ge 1}\bigl\{i \bigm| \chi(I^i)\ne\id\bigr\}.
$$
We set $m_{\triv}=-1$.
\end{notation}

\begin{lemma}
\label{lemmchi}
Let $\chi$ and $\rho$ be irreducible continuous representations 
of $I_\K$.
\begin{enumerate}
\item 
$
  m_\chi=\frac{f_\chi}{\dim\chi}-1.
$
\item
$\chi$ factors through $I_\K/I^n$ if and only if $m_\chi<n$.
\item If $m_\rho<m_\chi$, 
  then $f_{\rho\tensor\chi}=f_\chi\dim\rho$.
\item If $m_\rho=m_\chi$, then 
  $f_{\rho\tensor\chi}\le f_\chi\dim\rho=f_\rho\dim\chi$.
\end{enumerate}
\end{lemma}

\begin{proof}
(1) See \cite{SerL} Ch. 6, \S2, Exc.~2. 
(This is formulated in \cite{SerL} on the level of finite extensions, but
the definition extends directly to the whole of $I_\K$, since upper 
ramification groups behave well under quotients.)

(2) Clear.

(3) By (2), the group $I^{m_\chi}$ acts trivially on 
$\rho$ and non-trivially on $\chi$. As $I^{m_\chi}\normal I_K$
and $\chi$ is irreducible, $\chi$ has no invariants under $I^{m_\chi}$. 
Therefore $\rho\tensor\chi$
has no invariants either, and so each irreducible constituent $\tau$ of 
$\rho\tensor\chi$ has $m_\tau\ge m_\chi$. On the other hand, clearly $I^m$
acts trivially on $\rho\tensor\chi$ for $m>m_\chi$, so $m_\tau=m_\chi$
and $f_\tau=\frac{f_\chi}{\dim\chi}\dim\tau$ by (1). 
Taking the sum over $\tau$, we deduce that
$$
  f_{\rho\tensor\chi}=f_\chi\frac{\dim(\chi\tensor\rho)}{\dim\chi}=f_\chi\dim\rho.
$$

(4) As in (3), let $\tau$ be an irreducible constituent of $\rho\tensor\chi$.
For $m>m_\chi$, the group $I^m$ acts trivially 
on $\chi$ and on $\rho$, and therefore on $\tau$. Therefore, by (2), 
$m_\tau<m$, and taking the limit $m\to m_\chi$ gives $m_\tau\le m_\chi$.
By (1), $f_{\tau}\le\frac{f_\chi}{\dim\chi}\dim\tau$, and taking the sum 
over $\tau$ as in (3) gives the claim.
\end{proof}

\begin{theorem}
\label{condmain}
Let $\K$ be an $l$-adic field, and $\rho: I_\K\to\GL_d(\C)$ a 
non-trivial continuous irreducible representation. 
Take $n>\frac{f_\rho}{\dim\rho}-1$, and write
$N=I^n$ for the $n$th ramification subgroup (in the upper numbering)
of $I_\K$.
Then for every finite extension $\F/\K$, 
$$
  f_{\Res_\F\rho} \>=\> f_{\Res_{\F^N}\rho} \>\le\> e_{\F^N/\K} (f_\rho-d)+d.
$$
\end{theorem}

\begin{proof}
Because conductors and the upper numbering remain unchanged 
in unramified extensions, 
we may assume that $\F/\K$ is totally ramified. 
Decompose $\Ind_{\F/\K}\triv=\bigoplus_{\tau\in J}\tau$, with $\tau$ irreducible.
By the conductor-discriminant formula,
$$
\begin{array}{llllll}
  f_{\Res_\F\rho}  
       &=& f_{\Ind_{\F/\K}\Res_\F\rho} -  d\,v_\K(\Delta_{\F/\K}) \\[4pt]
       &=& \displaystyle f_{\rho\tensor\Ind_{\F/\K}\triv} -  d\,f_{\Ind_{\F/\K}\triv} 
       =\sum_{\tau\in J} f_{\rho\tensor\tau} - d\,f_{\tau}.\cr
\end{array}
$$
Let $J_N\subset J$ be the (multi-)set of $\tau$ which factor through $I_\K/N$.
Then
$$
  f_{\Res_{\F^N}\rho} =   \sum_{\tau\in\smash{J_N}} f_{\rho\tensor\tau} - d\,f_{\tau},
$$
by the same argument. However, $m_\tau\!\ge\! n\!>\!m_\rho$ for $\tau\notin J_N$ by 
Lemma \ref{lemmchi}(2), 
and so the terms $f_{\rho\tensor\tau} - d\,f_{\tau}$ are 0 by Lemma \ref{lemmchi}(3)
for such $\tau$. This proves the equality $f_{\Res_\F\rho} = f_{\Res_{\F^N}\rho}$.

For the inequality, first observe that by the argument above,
$$
  f_{\Res_\F\rho}  =       
    \sum_{\smash{J_\rho}} (f_{\rho\tensor\tau} - d\,f_{\tau}),
$$
where $J_\rho\subset J_N$ is the set of $\tau\in J$ for which $m_\tau\le m_\rho$.
By Lemma \ref{lemmchi}(3),(4), noting that 
$\dim\tau\le f_\tau$ for every $\tau\ne\triv$, we have
$$
\begin{array}{llllll}
  f_{\Res_\F\rho}  
       &=& \displaystyle
       \sum_{\tau\in J_\rho} (f_{\rho\tensor\tau} \!-\! d\,f_{\tau}) \le
       \sum_{\tau\in J_\rho} (f_\rho\dim\tau  \!-\! d\,f_{\tau})        \cr
       &\le& \displaystyle
       f_\rho + \sum_{\tau\in J_\rho\setminus\{\triv\}} (f_\rho\dim\tau  - d\,\dim\tau) \cr
       &=& \displaystyle
         f_\rho + (f_\rho-d)\bigl(\bigl(\sum_{\tau\in J_\rho}\dim\tau\bigr)-1\bigr) \cr
       &\le& \displaystyle
         f_\rho + (f_\rho-d)\bigl(\bigl(\sum_{\tau\in J_N}\dim\tau\bigr)-1\bigr) \cr
       &=& f_\rho + (f_\rho-d)([\F^N:\K]-1) 
       = d + (f_\rho-d)[\F^N:\K],
\end{array}
$$
which gives the claim, as we assumed that $\F/\K$ is totally ramified.
\end{proof}

\begin{theorem}
\label{ellcond}
Let $\K$ be an $l$-adic field, and $E/\K$ an elliptic curve with
additive reduction.
Take $n>\frac{f_E}{2}-1$ and write
$N=I^n$ for the $n$th ramification subgroup (in the upper numbering)
of $\Gal(\bar\K/\K)$.
Then for every finite extension $\F/\K$, 
$$
  f_{E/\F} \>=\> f_{E/\F^N} \>\le\> e_{\F^N/\K} (f_{E/\K}-2)+2.
$$
\end{theorem}

\begin{proof}
Write $\L=\F^N$.
Let $V$ be the Weil-Deligne representation associated to the first \'etale
cohomology of $E/\K$, and $\rho$ its `Weil part' 
(i.e. the semisimplification of $H_{\text{\'et}}^1(E,\Q_\ell)\tensor\C$ for some $\ell\ne l$). 

Either $\rho$ is irreducible or it is a sum of two one-dimensional characters 
$\psi_1, \psi_2$ of the same conductor (because $\psi_1\psi_2=\det \rho$ is
the cyclotomic character, which is unramified). 
The conductor depends only on the action of the inertia group, so, 
applying Theorem \ref{condmain} either to $\rho$ directly or to the $\psi_i$,
we find that $f_{\Res_\F\rho}=f_{\Res_\L\rho}$. 

Now, if $E/\K$ has potentially good reduction, 
then $V$ is a Weil representatation and $f_E=f_\rho$, 
so the result follows from Theorem \ref{condmain}.
If $E/\L$ has multiplicative reduction, then $f_{E/\F}\!=\!f_{E/\L}\!=\!1$,
and the result again follows.

Finally, suppose $E/\L$ has additive potentially multiplicative reduction.
By Theorem \ref{condmain}, 
$$
  f_{\Res_\F\rho} \>=\> f_{\Res_{\L}\rho} \>\le\> e_{\L/\K} (f_\rho-2)+2.
$$
Over any field where $E$ has additive reduction, $f_E=f_\rho\ge 2$.
This fact over $\L$ together with the formula shows that $f_{\Res_\F\rho}\ge 2$,
in particular $E$ has additive reduction over $\F$.
Thus $E$ has additive reduction over all three fields, and the 
formula then translates to the claim in the theorem.
%
%
\end{proof}

\begin{corollary}
\label{condell}
Let $\K$ be an $l$-adic field, and $\K^\infty/\K$ a possibly infinite
Galois extension. Let $E/\K$ be an elliptic curve with additive reduction.
Suppose that for some $n\!>\!\frac{f_E}{2}\!-\!1$,
the upper ramification subgroup $I^n_{\K_\infty/\K}$ of the inertia group of $\K^\infty/\K$ 
has finite index $e$.
Then for every finite extension $\L$ of $\K$ in $\K^\infty$,
$$
  f_{E/\L} \le e (f_{E/\K}-2)+2.
$$
\end{corollary}

\begin{proof}
Clear, since $e_{\L^{I_n}/\K}=e_{\L^I/\K}\le e$, where $I=I^n_{\K_\infty/\K}$.
\end{proof}

\begin{corollary}
\label{corliestab}
Let $\K$ be an $l$-adic field, and $\K^\infty/\K$ a Galois extension 
whose Galois group is an $l$-adic Lie group. 
Then there is a constant $C>0$ such that every elliptic curve $E/\K$ 
has conductor exponent $f_{E/\L}\le C$ over all finite extensions $\L$ 
of $\K$ in $\K^\infty$.
Moreover, if $\K^\infty=\bigcup_m \K_m$ with $\K_m/\K$ finite Galois,
then $f_{E/\K_m}$ stabilises as $m\to\infty$.
\end{corollary}

\begin{proof}
This is clear for curves with good and multiplicative reduction.
If $E/\K$ has additive reduction, by \cite{Pap} Thm. 1 the 
conductor exponent $f_{E/\K}$ is at most $2$, $3v_\K(3)+2$ and $6v_\K(2)+2$ 
when $l\ge 5$, $l=3$ and $l=2$, respectively. 
By Sen's theorem (\cite{Sen} \S4, main Theorem), 
$I^n_{\K^\infty/\K}\normal I_{\K^\infty/\K}$ 
is of finite index for every~$n$, so the two claims follow from 
Corollary \ref{condell} and Theorem \ref{ellcond}.
\end{proof}


\begin{thebibliography}{9}

\bibitem{BD}
A. Betts, V. Dokchitser,
Finite quotients of $\Z[C_n]$-lattices and Tamagawa numbers
of abelian varieties, preprint, 2012.

\bibitem{Magma}
W. Bosma, J. Cannon, C. Playoust,
The Magma algebra system. I: The user language, J. Symb. Comput. 24, No. 3--4 (1997),
235--265.

\bibitem{CasVIII}
J. W. S. Cassels,
Arithmetic on curves of genus 1. VIII: On conjectures of Birch and Swinnerton-Dyer,
J. Reine Angew. Math. 217 (1965), 180--199 (1965).

\bibitem{CS}
J. Coates, R. Sujatha, 
On the $\cM_H(G)$-conjecture, Non-abelian fundamental groups and Iwasawa theory, 
London Math. Soc. Lecture Note Ser., 393, Cambridge Univ. Press, Cambridge, 2012,
132--161.

\bibitem{CFKSV}
J. Coates, T. Fukaya, K. Kato, R. Sujatha, O. Venjakob,
The GL$_2$ main conjecture for elliptic curves without
complex multiplication, 
Publ. Math. IHES 101 (2005), 163--208.

\bibitem{iwacomp}
T. Dokchitser, V. Dokchitser, Computations in non-commutative Iwasawa theory, 
with app. by J. Coates and R. Sujatha, 
Proc. London Math. Soc. (3) 94 (2006), 211--272.

\bibitem{squarity}
T. Dokchitser, V. Dokchitser,
On the Birch--Swinnerton-Dyer quotients modulo squares,
Annals of Math. 172 no. 1 (2010), 567--596.

\bibitem{kurast}
T. Dokchitser, V. Dokchitser,
Root numbers and parity of ranks of elliptic curves,
J.~Reine Angew. Math. (Crelle), Number 658 (2011), 39--64.

\bibitem{isogloc}
T. Dokchitser, V. Dokchitser, Local invariants of isogenous elliptic curves,
preprint, August 2012, arxiv: 1208.5519.

\bibitem{tate}
T. Dokchitser, V. Dokchitser, A remark on Tate's algorithm and Kodaira types,
preprint, September 2012, arxiv: 1209.1414.

\bibitem{DriS}
M. Drinen, Iwasawa $\mu$-invariants of elliptic curves and their
symmetric powers, J.~Number Theory 102 (2003), 191--213.

\bibitem{Gre}
R. Greenberg, Introduction to Iwasawa theory for elliptic curves,
in Arithmetic al\-gebraic geometry, B. Conrad (ed.) et al.,
IAS/ Park City Math. Ser. 9, 407--464 (2001).

\bibitem{Gre2}
R. Greenberg, 
Iwasawa theory for elliptic curves, in Arithmetic theory of elliptic curves 
(Cetraro, 1997), Lecture Notes in Math. 1716 (1999), Springer, Berlin, 51--144.

\bibitem{HM}
Y. Hachimori, K. Matsuno, An analogue of Kida's formula for Selmer groups
of elliptic curves, J. Algebraic Geom. 8 (1999), 581--601.

\bibitem{HV}
Y. Hachimori, O. Venjakob, Completely faithful Selmer groups over Kummer
extensions, Documenta Mathematica, Extra Volume: Kazuya Kato's Fiftieth
Birthday (2003), 443--478.

\bibitem{KatP}
K. Kato, p-adic Hodge theory and values of zeta functions of modular
forms, in Co\-homologies $p$-adique et applications arithmetiques III,
Asterisque 295 (2004), ix, 117--290.

\bibitem{Kra}
A. Kraus, Sur le d\'efaut de semi-stabilit\'e des courbes elliptiques
\`a r\'eduction additive, Manuscripta Math. 69 (1990), no. 4, 353--385.

\bibitem{KT}
K. Kramer, J. Tunnell, Elliptic curves and local $\epsilon$-factors,
Compositio Math. 46 (1982), 307--352.


\bibitem{Kur}
M. Kurihara, On the Tate-Shafarevich groups over cyclotomic fields 
of an elliptic curve with supersingular reduction. I, 
Invent. Math. 149, No. 1 (2002), 195--224.

\bibitem{Lam}
J. Lamplugh, An analogue of the Washington--Sinnott theorem 
for elliptic curves with complex multiplication, in preparation.

\bibitem{LRS}
P. Lockhart, M. Rosen, J. Silverman, An upper bound for the conductor 
of an abelian variety, J. Algebraic Geom. 2, no. 4 (1993), 569--601.

\bibitem{Pap}
I. Papadopolous, Sur la classification de N\'eron des courbes elliptiques,
J. Number Theory 44, no. 2 (1993), 119--152.

\bibitem{PeV}
B. Perrin-Riou, Variation de la fonction $L$ $p$-adique par isog\'enie,
Adv. Stud. in Pure Math. 17 (1989), 347--358.

\bibitem{Rib}
K. A. Ribet, Torsion points of abelian varieties in cyclotomic extensions,
Appendix to N. M. Katz and S. Lang, Finiteness theorems in geometric 
class field theory, L'Enseignement Math. 27 (1981), 315--319. 

\bibitem{Sch}
P. Schneider, The $\mu$-invariant of isogenies,
J. Indian Math. Soc., New Ser. 52 (1987), 159--170.

\bibitem{Sen}
S. Sen, Ramification in $p$-adic Lie extensions, Invent. Math. 17 (1972), 44--50.

\bibitem{SerP}
J.-P. Serre, Propri\'et\'es galoisiennes des points d'ordre fini des courbes
elliptiques, Invent. Math. 15 (1972), no. 4, 259--331.

\bibitem{SerL}
J.-P. Serre, Local fields, Graduate Texts in Mathematics, vol. 67, 
Springer-Verlag, New York, 1979.

\bibitem{Sil1}
J. H. Silverman, The Arithmetic of Elliptic Curves,
GTM 106, Springer-Verlag 1986.

\bibitem{Sil2}
J. H. Silverman, Advanced Topics in the Arithmetic of Elliptic Curves,
Graduate Texts in Mathematics 151, Springer-Verlag 1994.

\bibitem{TatC}
J. Tate, On the conjectures of Birch and Swinnerton-Dyer and a
geometric analog, S\'eminaire Bourbaki, 18e ann\'ee, 1965/66, no. 306.

\bibitem{Was}
L. Washington,
The non-$p$-part of the class number in a cyclotomic $\Z_p$-extension,
Invent. Math. 49, issue 1 (1978), 87--97.

\end{thebibliography}
\end{document}